\documentclass[12pt,reqno]{amsart}
\usepackage[headings]{fullpage}
\usepackage{amssymb,amsmath,mathtools,bbm,tikz,tikz-cd,color,relsize,multirow}
\usepackage{stmaryrd}
\usepackage[all,cmtip]{xy}
\usepackage{url}
\usetikzlibrary{positioning}
\usetikzlibrary{calc}
\usetikzlibrary{decorations.markings}
\usetikzlibrary{arrows}
\usetikzlibrary{calc}
\usetikzlibrary{decorations.markings}
\tikzstyle{hvector}=[inner sep=2pt,draw=blue!50,fill=blue!10,thick]
\tikzstyle{unit}=[inner sep=2pt,shape=circle, draw]
\tikzstyle{counit}=[inner sep=2pt,shape=circle, draw,fill=gray]
\tikzstyle{antipode}=[inner sep=2pt,shape=rectangle, draw]
\tikzstyle{cocycle}=[inner sep=2pt,shape=circle, draw]
\tikzstyle{twistedm}=[inner sep=2pt,shape=circle, fill=gray]
\tikzstyle{autom}=[inner sep=2pt,shape=circle, draw]
\tikzstyle{coact}=[inner sep=2pt,shape=circle, fill=black]

\tikzset{> /.tip = {Stealth[round,length=6pt]}}
\tikzstyle over=[preaction={draw,line width=5pt,white}]
\tikzset{baseline={([yshift=-.5ex]current bounding box.center)} }
\tikzset{every path/.style={thick}
}

\allowdisplaybreaks

\usepackage[bookmarks=true,%
    colorlinks=true,%
    linkcolor=blue,%
    citecolor=blue,%
    filecolor=blue,%
    menucolor=blue,%
    urlcolor=blue,%
    breaklinks=true]{hyperref}
\usepackage{slashed} 
\usepackage{verbatim}
\usepackage[normalem]{ulem}
\usepackage{diagbox}

\newtheorem{theorem}{Theorem}[section]
\theoremstyle{definition}
\newtheorem{proposition}[theorem]{Proposition}
\newtheorem{lemma}[theorem]{Lemma}
\newtheorem{definition}[theorem]{Definition}
\newtheorem{remark}[theorem]{Remark}
\newtheorem{corollary}[theorem]{Corollary}
\newtheorem{conjecture}[theorem]{Conjecture}

\def\red#1{\textcolor{red}{[#1]}}

\definecolor{green}{rgb}{0.0,0.5,0.0}

\def\BZ{\mathbbm Z}
\def\BQ{\mathbbm Q}

\def\BC{\mathbbm C}
\def\BK{\mathbbm K}

\def\Bunit{\mathbbm{1}}

\def\calI{\mathcal I}

\def\calC{\mathcal C}

\def\calP{\mathcal P}

\def\la{\langle}
\def\ra{\rangle}

\def\tq{\tilde{q}}

\def\a{\alpha}
\def\b{\beta}
\def\g{\gamma}

\def\th{\theta}

\def\be{\begin{equation}}
\def\ee{\end{equation}}

\def\maxdeg{\mathrm{maxdeg}}

\def\sG{\mathsf{G}}

\def\sX{\mathsf{X}}
\def\vphi{\varphi}

\def\red#1{\textcolor[rgb]{1.00,0.00,0.00}{#1}}

\def\End{\mathrm{End}}
\def\id{\mathrm{id}}
\def\trl{\mathrm{tr}_L}
\def\trr{\mathrm{tr}_R}
\def\ptr{\mathrm{ptr}}

\def\LG{\mathrm{LG}}
\def\genus{\mathrm{genus}}
\def\br#1{ \{ #1 \}}

\def\Alex{\Delta}

\def\p#1{\ensuremath{\overline {#1}}}
\def\UqG{\mathrm{U}_q\mathfrak{sl}(2|1)}
\def\Uq{\mathrm{U}_q}
\def\UqGH{\mathrm{U}_q^H\mathfrak{sl}(2|1)}
\def\UqG{\mathrm{U}_q\mathfrak{sl}(2|1)}
\def\UqH{\mathrm{U}_q^H}
\def\Hom{\mathrm{Hom}}

\def\cat{\calC}
\def\catsl{\calC^{\mathfrak{sl}(2|1)}}
\def\proj{\calP}
\def\unit{\Bunit}
\def\knot{K}
\def\link{L}
\def\tangle{T}
\def\sl21{\mathfrak{sl}(2|1)}
\def\U{\mathrm{U}}

\newcommand{\lcoev}{\stackrel{\longleftarrow}{\operatorname{coev}}}
\newcommand{\lev}{\stackrel{\longleftarrow}{\operatorname{ev}}}
\newcommand{\rev}{\stackrel{\longrightarrow}{\operatorname{ev}}}
\newcommand{\rcoev}{\stackrel{\longrightarrow}{\operatorname{coev}}}
\newcommand{\mt}{\mathsf{tr}}
\newcommand{\md}{\mathsf{d}}
\newcommand\mRT{F'}
\newcommand\RT{F}

\newcommand{\drawproj}[3]{%  #1: x-coordinate, #2: y-coordinate, #3: color 
  \draw[] (#1,#2) to[in=-150,out=90] (#1+.5,#2+.35);
  \draw[] (#1+1,#2) to[in=-30,out=90] (#1+.5,#2+.35);	
  \draw[very thick,#3](#1+.5,#2+.35) to (#1+.5,#2+.85);
  \draw[->] (#1+.5,#2+.85) to[out=150,in=-90] (#1,#2+1.3);
  \draw[->] (#1+.5,#2+.85) to[out=30,in=-90] (#1+1,#2+1.3);	
}

\newcommand{\drawprojc}[1]{
    \draw[] (0,0) to[in=-150,out=90] (.5,.35);
	\draw[] (1,0) to[in=-30,out=90] (.5,.35);	
	\draw[very thick,#1](.5,.35) to (.5,.85);
	\draw[] (.5,.85) to[out=150,in=-90] (0,1.3);
	\draw[] (.5,.85) to[out=30,in=-90] (1,1.3);	
    \draw (1,0) arc (180:360:.5);
    \draw (1,2.3) arc (180:0:.5);
    \draw[] (2,1.3) to[out=90, in=-90] (1,2.3); 
    \draw[over,->] (1,1.3) to[out=90, in=-90] (2,2.3);
    \draw (2,0) to (2,1.3);
    \draw (0,0) to (0,-.5);
    \draw[->] (0,1.3) to (0,2.3);
}

\makeatletter

\renewcommand\thepart{\@Roman\c@part}%
\renewcommand\part{%
   \if@noskipsec \leavevmode \fi
   \par
   \addvspace{6.7ex}%
   \@afterindentfalse
   \secdef\@part\@spart}
\def\@part[#1]#2{%
    \ifnum \c@secnumdepth >\m@ne
      \refstepcounter{part}%
      \addcontentsline{toc}{part}{Part~\thepart.\ #1}%
    \else
      \addcontentsline{toc}{part}{#1}%
    \fi
    {\parindent \z@ \raggedright
     \interlinepenalty \@M
     \normalfont
     \ifnum \c@secnumdepth >\m@ne
       \centering\large\scshape \partname~\thepart.%
       \hspace{1ex}%
     \fi%
     \large\scshape #2%
     \markboth{}{}\par}%
    \nobreak
    \vskip 4.7ex
    \@afterheading}
  \def\@spart#1{
  \refstepcounter{part}%
  \addcontentsline{toc}{part}{#1}%
    {\parindent \z@ \raggedright
     \interlinepenalty \@M
     \normalfont
     \centering\large\scshape #1\par}%
     \nobreak
     \vskip 4.7ex
     \@afterheading}
\renewcommand*\l@part[2]{%
  \ifnum \c@tocdepth >-2\relax
    \addpenalty\@secpenalty
    \addvspace{0.75em \@plus\p@}%
    \begingroup
      \parindent \z@ \rightskip \@pnumwidth
      \parfillskip -\@pnumwidth
      {\leavevmode
       \normalsize \bfseries #1\hfil \hb@xt@\@pnumwidth{\hss #2}}\par
       \nobreak
       \if@compatibility
         \global\@nobreaktrue
         \reverypar{\global\@nobreakfalse\reverypar{}}%
      \fi
    \endgroup
  \fi}

\def\l@subsection{\@tocline{2}{0pt}{2pc}{6pc}{}}
\makeatother

\begin{document}

\title[On the colored Links--Gould polynomial]{
  On the colored Links--Gould polynomial}

\author[S. Garoufalidis]{Stavros Garoufalidis}
\address{% Max Planck Institute for Mathematics \\
         % Vivatsgasse 7, 53111 Bonn, GERMANY \newline
  International Center for Mathematics, Department of Mathematics \\
  Southern University of Science and Technology \\
  Shenzhen, China \newline
  {\tt \url{http://people.mpim-bonn.mpg.de/stavros}}}
\email{stavros@mpim-bonn.mpg.de}

\author[M. Harper]{Matthew Harper}
\address{Michigan State University, East Lansing, Michigan, USA}
\email{mrhmath@proton.me}

\author[R. Kashaev]{Rinat Kashaev}
\address{Section de Math\'ematiques, Universit\'e de Gen\`eve \\
rue du Conseil-G\'en\'eral 7-9, 1205 Gen\`eve, Switzerland \newline
         {\tt \url{http://www.unige.ch/math/folks/kashaev}}}
\email{Rinat.Kashaev@unige.ch}
       
\author[B.-M. Kohli]{Ben-Michael Kohli}
\address{Section de Math\'ematiques, Universit\'e de Gen\`eve \\
rue du Conseil-G\'en\'eral 7-9, 1205 Gen\`eve, Switzerland}
\email{bm.kohli@protonmail.ch}

\author[E. Wagner]{Emmanuel Wagner}
\address{Univ Paris Cit\'e, IMJ-PRG, Univ Paris Sorbonne, UMR 7586 CNRS, F-75013, Paris,
France}
\email{wagner@imj-prg.fr}


\date{17 December 2025}%{\today }
%\dedicatory{}

\begin{abstract}
We give a cabling formula for the Links--Gould polynomial of knots colored with a
  $4n$-dimensional irreducible representation of $\mathrm{U}^H_q\sl21$
  and identify them with the $V_n$-polynomial of knots for $n=2$. Using the
  cabling formula, we obtain genus bounds and a specialization to the Alexander
  polynomial for the colored Links--Gould polynomial that is independent of $n$,
  which implies corresponding properties of the $V_n$-polynomial for $n=2$ conjectured
  in previous work of two of the authors, and extends the work done for $n=1$.
  Combined with work of one of the authors \cite{GL:patterns}, our genus bound
  for $\LG^{(2)}=V_2$ is sharp for all knots with up to $16$ crossings.
\end{abstract}

\maketitle

{\footnotesize
\tableofcontents
}

%%%%%%%%%%%%%%%%%%%%%%%%%%%%%%%%%%%%%%%%%%%%%%%%%%%%%%%%%%%%%%%%%%%%%%%%%%%%
%%%%%%%%%%%%%%%%%%%%%%%%%%%%%%%%%%%%%%%%%%%%%%%%%%%%%%%%%%%%%%%%%%%%%%%%%%%%

\section{Introduction}
\label{sec.intro}

The paper concerns the identification of two sequences of 2-variable knot
polynomials, one that comes from the representation theory of
$\sl21$ (the colored Links--Gould polynomial of a knot),
and another that comes from the recent work of two of the authors~\cite{GK:multi}
using the $R$-matrix of a rank 2 Nichols algebra (the $V_n$-polynomials of a knot).  

Let us recall these knot polynomials starting from the representation theory of the
associated quantum groups. The latter are the unrolled quantum groups
of~\cite{Geer:superalgebra} whose simple weight modules are deformations 
of representations of the classical Lie superalgebra $\sl21$ by results of Kac and Geer
\cite{Kac:superalgebra, Geer:superalgebra}, and also \cite{Geer:multi,Geer:multi1}. These simple, finite dimensional weight representations are graded vector spaces
$V(n,\a)$  of dimension $4(n+1)$ with highest weight $(n,\a) \in \BZ_{\geq 0}
\times \BC$ and satisfy a (typicality) condition $\a(n+1+\a) \neq 0$. We also assume that the grading of the highest weight space of $V(n,\a)$ has parity $n \bmod 2$.

Using a Reshetikhin--Turaev construction, we let $\LG^{(n+1)}(q^\a,q)
\in \BZ[q^{\pm \a}, q^{\pm 1}]$
denote the knot polynomial associated to the irreducible representation $V(n,\a)$. In our notation, $\LG^{(1)}$ (usually denoted by $\LG^{2,1}$ or simply $\LG$)
is the Links--Gould invariant~\cite{Links-Gould} and
the sequence $\LG^{(n)}$ of 2-variable polynomials for $n \geq 1$ is the colored
Links--Gould
polynomial\footnote{
One may also denote these colored 2-variable polynomial invariants by $\LG^{2,1}_{n}$.}. Our first result is that the colored Links--Gould polynomial
of a knot is uniquely determined by the sequence of the $\LG^{(1)}$-polynomials of its
parallels. The relation between $\LG^{(1)}$ of the $(n,0)$-parallel of a
0-framed knot and $\LG^{(k)}$ for $k\leq n$ is given in the next theorem.

\begin{theorem}
\label{thm.n0}
For any 0-framed knot $\knot$ and all $n \geq 1$, we have
\be
\label{LGn0}
\LG^{(1)}_{\knot^{(n,0)}}(q^\a,q) = \sum_{k + \ell \leq n-1}
m^{(n)}_{k,\ell} A^{(n)}_{k,\ell}(q^\a,q) \LG^{(k+1)}_{\knot}(q^{n\a +\ell},q)
\ee
where $m^{(n)}_{k,\ell}$ is the multiplicity of the representation
$V(k,n\a+\ell)$ in $V(0,\a)^{\otimes n}$ and 
\be
\label{aijn}
A^{(n)}_{k,\ell}(q^\a,q) = 
(-1)^k\frac{\br {k+1} \br \a \br {\a+1} }{\br 1 \br {{n\a +\ell}}
  \br {{n\a+k+\ell+1}}} \,.
\ee
with $\br x=q^x-q^{-x}$. 
\end{theorem}  

The multiplicity $m^{(n)}_{k,\ell}$ has been studied in detail in~\cite{Anghel}
and can be computed recursively by
\be
\label{mrec}
m^{(n)}_{k,\ell} = m^{(n-1)}_{k,\ell} + m^{(n-1)}_{k,\ell-1} + m^{(n-1)}_{k-1,\ell}
+ m^{(n-1)}_{k+1,\ell-1}
\ee
with the understanding that $m^{(n)}_{k,\ell}=0$ if $k<0$ or $\ell<0$ and with
the initial condition $m^{(n)}_{n-1,0}=1$. 

The above theorem is a corollary of a more general statement for TQFTs associated
to categories of modules of quantum (super)groups, applying more generally to any
$\BK$-linear, locally-finite, unimodular, ribbon category $\cat$ with enough
projectives. We give a precise definition of these adjectives in Section~\ref{sec.proofn0}. Such a category $\cat$ has  a full subcategory $\calP$
of projective objects and a result of~\cite[Cor.5.6]{GKPM22} asserts that, up
to a global scalar, there is a unique modified trace on $\calP$. The modified trace determines a renormalized Reshetikhin--Turaev invariant
$\mRT_{V}$ of framed, oriented links in 3-space, whose components are colored by
elements of $\calP$. A detailed description of this is given in
\cite{GPT, GKPM11,GKPM22}. Throughout the paper, by link or tangle we always
mean a framed, oriented one.

\begin{theorem}
\label{thm.n0cable}
Fix a simple, projective object $V\in \mathcal{C}$ such that $V^{\otimes n}$ is
a direct sum of finitely many simple objects $V_i$. For all knots
$\knot$, we have
\be
\label{n0cable}
\mRT_{V,\knot^{(n,0)}} = \sum_{i}
\mRT_{V_i,\knot}\,.
\ee
\end{theorem}  

We will apply the above theorem in the category $\catsl$ defined in detail
in Section~\ref{sub.catsl21} below and to the simple, projective object $V(0,\a)$.
Roughly, the objects of $\catsl$ are finite dimensional weight-modules of
the unrolled quantum group of the Lie superalgebra $\sl21$
and the morphisms are linear maps that intertwine the action of the quantum group.
The category $\catsl$ is monoidal with respect to the tensor product of
representations, and is ribbon due to the ribbon Hopf algebra structure of
the unrolled quantum group. 

The Links--Gould invariants of an unframed link $\link$ are related to the
renormalized invariant of a 0-framed representative of $L$ by
\be
\mRT_{V(k,\alpha),\link}=\LG^{(k+1)}_{\link}(q^\a,q)\cdot\md(V(k,\alpha))
\ee
where, up to a global scalar,
\be
\label{eq:ModifiedDimension}
\md(V(k,\alpha))=(-1)^k\frac{\br{k+1}}{\br{{\a}}\br{{k+\alpha+1}}}
\ee
is the modified dimension of $V(k,\alpha)$ and equals the value of $F'$ on the unknot.
The coefficients $A^{(n)}_{k,\ell}(q^\a,q)$ in \eqref{aijn} satisfy
\be
\label{aijn2}
A^{(n)}_{k,\ell}(q^\a,q) = \frac{\md(V(k,n\alpha+\ell))}{\md(V(0,\alpha))}\red{.}
\ee
and are therefore independent of the scaling of the modified dimension. Theorem \ref{thm.n0} together with the symmetries, specializations, and genus
bounds for the Links--Gould polynomial $\LG^{(1)}$ imply
the following result, see~\cite{Ishii,Kohli,KPM, Kohli-Tahar}.

\begin{theorem}
\label{thm.LG}
The colored Links--Gould polynomial of knots satisfies the symmetries and
specializations
\be
\label{LGspec}
\LG^{(n)}(q^\a,q) = \LG^{(n)}(q^{-\a-n},q), \qquad
\LG^{(n)}(1,q) = 1, \qquad
\LG^{(n)}(q^\a,1) = \Alex(q^{2\a})^2 ,
\ee
where $\Alex$ is the Alexander-Conway polynomial, and the genus bounds
\be
\label{LGgenus}
\deg_{q^\a} \LG^{(n)}_\knot (q^{\a},q) \leq 8 \, \genus(\knot) 
\ee
hold for all $n \geq 1$ and knots $\knot$. 
\end{theorem}

The next sequence of polynomials appeared in recent work of two of the authors;
explicitly, $V_n(t,q) \in \BZ[t^{\pm 1}, q^{\pm n/2}]$ was the polynomial defined
in~\cite[Sec.7.3]{GK:multi}. In that work some conjectures were stated about the
symmetries, specializations, and genus bounds of the $V_n$-polynomials, as well as
a relation between $V_1$ and the Links--Gould polynomial. This, together with
our earlier work, suggests the following conjecture.

\begin{conjecture}
\label{conj.1}
For all $n \geq 1$, we have:
\be
\label{LGnVn}
\LG^{(n)}(q^\a,q) = V_n(q^{2\a+n},q^{2}) \in \BZ[q^{\pm 2 \a}, q^{\pm1}] \,.
\ee
\end{conjecture}
Conjecture~\ref{conj.1} was proven for $n=1$ by showing that $V_1$ and $\LG^{(1)}$
satisfy a common skein theory that uniquely determines them~\cite{GHKKST}. We emphasize that our convention for the polynomial variables $q^\a$ and $q$ here are inverse to the conventions of previous papers related to Links--Gould because of reorganization of the representation theory of $\U_q^H\mathfrak{sl}(2|1)$.

Regarding the case of $n>1$, Theorem~\ref{thm.n0cable} gives
a cabling formula for $\LG^{(n)}$ in terms of $\LG^{(1)}$. The
spectral decomposition of the image of the braid $\sigma_{n-1} \sigma_{n-2} \ldots \sigma_2 \sigma_1$ under the representation derived from the $R$-matrix of the $V_1$-polynomial gives a cabling
formula for $V_n$ in terms of $V_1$ with some coefficients that are independent
of the knot. When $n=2$, we can compute these coefficients explicitly and this
implies the following.

\begin{theorem}
\label{thm.1}
Conjecture~\ref{conj.1} holds for $n=2$.
\end{theorem}

Conjecture~\ref{conj.1} and Theorem~\ref{thm.LG} imply that all the polynomial
invariants $V_n$ of a knot are determined by (and conversely, determine) the
$V_1$-invariant of a knot and its parallels. Moreover, they imply that 
the $V_n$ polynomials of a knot satisfy the symmetries and specializations
\be
\label{Vnspec}
V_n(t,\tq) = V_n(t^{-1},\tq), \qquad
V_n(\tq^{n/2},\tq) = 1, \qquad
V_n(t,1) = \Alex(t)^2 
\ee
and that the genus bounds 
\be
\label{Vngenus}
\deg_t V_{n,\knot}(t,q) \leq 4 \, \text{genus}(\knot) 
\ee
hold for all $n \geq 1$ and all knots $\knot$. 

A reverse consequence of Conjecture~\ref{conj.1} is that
$\LG^{(n)}(q^\a,q) \in\mathbb{Z}[q^{\pm2\alpha},q^{\pm1}]$ (proved by Ishii~\cite{Ishii}
for $n=1$) even though the coefficients of the $R$-matrix for $\LG^{(n)}$ are in
$\mathbb{Z}[q^{\pm \alpha},q^{\pm1}]$. 

A remarkable aspect of the genus bound~\eqref{Vngenus} is that it is independent
of $n$, and although $V_n$ fails to detect mutation when $n=1$ (for instance on the
mutant pair of 11 crossing knots with trivial Alexander polynomial that have genera 2 and
3 respectively), extensive computations of the $V_2$-polynomial imply the
following result. 

\begin{proposition}[\cite{GL:patterns}]
When $n=2$, Equation \eqref{Vngenus} is an equality for all $1701936$ prime knots
with up to $16$ crossings, and for infinitely many pairs of mutant knots of the
Kinoshita-Terasaka family.
\end{proposition}

\begin{remark}
Note that there are three sets of variables used in the literature, namely
$(t_0,t_1)$ introduced by Ishii~\cite{Ishii}, $(q^\a,q)$ used in the
context of representation theory e.g. to study $\LG$~\cite{Links-Gould, Kohli-Tahar}
and $(t,\tq)$ used in~\cite{GK:multi}. This is a point that leads to much confusion.
The relations between these different sets of variables are
\be
\label{rosetta}
(t_0, t_1)= (q^{2\a}, q^{-2(\a+1)}), \qquad
(t, \tq)= (q^{2\a+1}, q^{2}), \qquad (q^\a,q) = (t^{1/2} \tq^{-1/4}, \tq^{1/2}) \,. 
\ee
With these changes of variables, Equation~\eqref{LGnVn} becomes

\be
\label{GLnVn}
\LG^{(n)}(q^\a,q) = V_n(t \tq^{(n-1)/2}, \tq) \,.
\ee
\end{remark}

Throughout this article we always consider $\LG^{(n)}$ as a Laurent polynomial in the
variables $(q^\a,q)$ and $V_n$ as a Laurent polynomial in the variables $(t,\tq)$.\\

\subsection*{Acknowledgements}
~The authors thank John Links and Daniel L\'opez Neumann for comments on a previous version of this paper.
MH was partially supported through the NSF-RTG grant
\#DMS-2135960. The work of RK is partially supported by the SNSF research program NCCR The Mathematics of Physics (SwissMAP), and the SNSF grants no. 200021-232258 and no. 200020-200400. BMK was partially supported through the BJNSF grant \#IS24066.

%%%%%%%%%%%%%%%%%%%%%%%%%%%%%%%%%%%%%%%%%%%%%%%%%%%%%%%%%%%%%%%%%%%%%%%%%%%%
%%%%%%%%%%%%%%%%%%%%%%%%%%%%%%%%%%%%%%%%%%%%%%%%%%%%%%%%%%%%%%%%%%%%%%%%%%%%

\section{Basics on ribbon categories}
\label{sec.proofn0}

\subsection{Preliminaries}
\label{sub.cat}
In this section we review some basic notions on categories that can be found in~\cite{EGNO}.

To begin with, a (small) category $\cat$ consists of a set $\calC_0$ of objects $\{V\}$,
a set $\cat_1$ of morphisms $\{f: V \to W\}$, identity elements
$\{ \mathrm{id}_V: V \to V\}$ and an associative composition law
$f, g \mapsto g \circ f$ for morphisms $V \stackrel{f}{\to} W$ and
$W \stackrel{g}{\to} U$. 

If $\BK$ is a field, a category is \emph{$\BK$-linear} if the
set of morphisms between two objects is a $\BK$-vector space and composition of
morphisms is $\BK$-bilinear.

A category $\cat$ is \emph{monoidal} if it is equipped with a tensor product functor
$\otimes:\cat\times\cat\to\cat$, a monoidal unit
$\unit \in \cat$, as well as unitor and associator maps satisfying certain coherence
axioms \cite[\S2.1]{EGNO}.
Here we assume that $\otimes$ is $\BK$-bilinear and the
$\BK$-algebra map
$\BK \to \End_{\cat}(\unit), k \mapsto k \cdot \id_\unit$ is an isomorphism.

A monoidal category $\cat$ is \emph{rigid} if for each object $V\in\cat$ there is
a dual object $V^*\in\cat$. The duality structure maps are denoted
\be
\begin{aligned}
\lev_V&:V^{\ast} \otimes V \to \unit,
& \lcoev_V&:\unit \to V  \otimes V^{\ast}\,,
\\ \rev_V&:V\otimes V^{\ast}  \to \unit,
& \rcoev_V&: \unit \to V^{\ast}  \otimes V\,.
\end{aligned}
\ee

A rigid category $\cat$ is \emph{pivotal} if there exists a natural isomorphism
$\vphi=\br{\vphi_V:V\to(V^*)^*}$.

A monoidal category $\cat$ is \emph{braided} if there exists a natural isomorphism
$c=\{c_{V,W}: V \otimes W \rightarrow W \otimes V\}_{V,W \in \cat}$ called a braiding
satisfying two hexagon relations \cite[\S8.1]{EGNO}.
In essence, the hexagon relations state that braiding $V$ with two objects
sequentially is the same as braiding $V$ with their tensor product.

A braided rigid category $\cat$ is \emph{ribbon} if there exists a natural
isomorphism $\theta=\{\theta_V: V \rightarrow V\}_{V \in \cat}$ called a twist which
satisfies the relations
\begin{align}
\th_{V\otimes W} = c_{W,V}\circ c_{V,W}\circ (\th_V\otimes \th_W)\,,&&
\th_{\unit}=\id_{\unit}\,,&&
\th_{V^*}=\th^*_V\,.
\end{align}
Note that a ribbon category is necessarily pivotal.

\subsection{Indecomposable, simple and projective objects}
\label{sub.proj}

In this section, we discuss the indecomposable, simple, and
projective objects of a $\BK$-linear ribbon category $\cat$.

An object $W$ is a \emph{direct sum} if there exist objects and morphisms $\{(V_i,f_i,g_i)\}_{i\in I}$ where $f_i:V_i\to W$ and
$g_i:W\to V_i$ such that $\id_W= \sum f_i\circ g_i$ and
$g_i\circ f_j = \delta_{ij}\cdot \id_{V_i}$. We then write $W\cong \bigoplus V_i$.

An object $Q$ is a \emph{subobject} of $P$ if it is isomorphic to the monomorphic
image of $Q\hookrightarrow P$. We will typically identify $Q$ with its image in $P$.

An object $0\in \cat$ is a \emph{zero object} if $\Hom(0,V)=\Hom(V,0)=\br{0}$, i.e. the zero vector space, for every object $V\in\cat$. A zero object is essentially unique, with the unique isomorphism between any two zero objects given by the zero map. 

An object $W$ is \emph{indecomposable} if there do not exist non-zero subobjects $U$ and $V$ of $W$ such that $W\cong U\oplus V$. In other words, $W$ is indecomposable if whenever $W$ is isomorphic to a direct sum of $U$ and $V$, then at least one of $U$ or $V$ is the zero object.

An object $V$ of a $\BK$-linear category $\cat$ is \emph{simple} if
$\End_\cat(V)\cong\BK \cdot \id_V$. If $V$ is a simple object and $f\in\End_\cat(V)$, then
we may write $f=\la f\ra \cdot \id_V$\,.  Note that a simple object is indecomposable but not conversely. 

An object $W$ is \emph{semisimple} if it is isomorphic to a direct sum of simple objects.

A $\BK$-linear category $\cat$ is \emph{semisimple} if every object of $\cat$ is semisimple.

We call a $\BK$-linear category $\cat$ \emph{locally-finite} if for every pair of objects
$U,V \in \cat$, $\Hom_{\cat} (U,V )$ is a finite dimensional
$\BK$-module \emph{and} every object has a finite length composition series. 

An object $V$ in $\cat$ is \emph{projective} if for every epimorphism
$f: W \twoheadrightarrow Z$, and every morphism $g: V \to Z$, there exists a (not-necessarily unique)
lift $g' : V \to W$ such that $g=f \circ g'$.

We say that $\cat$ \emph{has enough projectives} if every object is the epimorphic
image of a projective object. The projective cover of an object $V$ is a pair
$(P,p)$ where $P$ is a projective object and $p:P\twoheadrightarrow V$ such that
$\ker p$ is \emph{superfluous}, i.e. for any subobject $Q$ of $P$, 
\begin{align}
\ker p + Q=P \quad \mbox{implies} \quad Q=P\,. 
\end{align}
If $V$ is simple, then $P$ is indecomposable and has a unique simple subobject which
we call the \emph{socle} of $P$.

Finally, we say that $\cat$ is \emph{unimodular} if the socle of the projective
cover of $\unit$ is isomorphic to $\unit$. 

\subsection{Modified trace}
\label{sub.modified}

In this section we discuss the quantum dimension and the modified dimension of
special objects of a $\BK$-linear ribbon category $\cat$. Additional details can be
found in \cite{GKPM11,CGP,GPT}.

Fix objects $V , W \in \cat$. The \emph{quantum dimension} of $V$ is
$\mathrm{qdim}(V)=~\rev_V\circ \lcoev_V ~=~ \lev_V\circ\rcoev_V$. When $V$ is a
typical module of a Lie superalgebra, the quantum dimension is zero.

The \emph{right partial trace} along $W$ is the map
\be
\begin{aligned}
\ptr_W : \End_{\cat}(V \otimes W) & \rightarrow  \End_{\cat}(V) \\
f & \mapsto  (\id_V \otimes \rev_W) \circ (f \otimes \id_{W^{\ast}})
\circ (\id_V \otimes \lcoev_W).
\end{aligned}
\ee
For any endomorphism $f\in\End_\cat(V\otimes V)$, set
\be
\begin{aligned}
  \trl(f) &= (\lev_V \otimes \id_V ) \circ (\id_{V^*} \otimes f )
  \circ (\rcoev_V \otimes  \id_V ) \in \End_\cat(V)\,,\\
  \trr(f) &= (\id_V \otimes \rev_V ) \circ (f \otimes \id_{V^*})
  \circ (\id_V \otimes \lcoev_V ) \in\End_\cat(V)\,.
\end{aligned}
\ee
An object $V$ of $\calC$ is called \emph{ambidextrous} if $\trl(f) = \trr(f)$ for
all $f \in\End_\cat(V\otimes V)$.

A full subcategory $\calI \subset \cat$ is an {\em ideal} if it has the following
properties.
\begin{enumerate}
\item
  If $W\in \mathcal{I}$ and $V \in \cat$, then $W\otimes V \in \mathcal{I}$.
\item
  If $W \in \mathcal{I}$ and $V \in \cat$
  is a direct summand, then $V\in \mathcal{I}$.
\end{enumerate}

The full subcategory $\calP$ of projective objects of $\calC$ is an ideal. 
In fact, in~\cite[Prop.1.2]{Geer:invariant-trace}, the ideal $\calI_V$ generated
by a typical module of a Lie superalgebra was shown to be independent of $V$ and
to coincide with $\calP$.

A \emph{modified trace} on an ideal $\calI \subset \cat$ is a family of
$\BK$-linear functions 
\be
\mt=\{\mt_V:\End_\cat(V) \rightarrow \BK \mid V \in \mathcal{I} \}
\ee
with the following properties.
\begin{enumerate}
\item \emph{Cyclicity}:
  For all $V,W \in \mathcal{I}$ and $f \in \Hom_{\cat}(W,V)$ and
  $g \in \Hom_{\cat}(V,W)$, there is an equality $\mt_V(f \circ g)=\mt_W(g \circ f)$.
\item \emph{Partial trace property}:
  For all $V \in \mathcal{I}$, $W \in \cat$ and $f\in\End_{\cat}(V\otimes W)$,
  there is an equality $\mt_{V\otimes W}(f)=\mt_V(\ptr_W(f))$.
\end{enumerate}

The \emph{modified dimension} of $V\in \mathcal{I}$ is $\md(V)=\mt_V(\id_V)$.
While the quantum dimension of a typical Lie superalgebra is zero, its modified
dimension is not.

\subsection{Proof of Theorem \ref{thm.n0cable}}

Fix a $\BK$-linear, locally-finite, unimodular, ribbon category $\cat$
with enough projectives. Let $\proj\subset \cat$ denote the ideal of
projective objects.

\begin{proposition}[{\cite[Cor.5.6]{GKPM22}}]\label{prop.normalize}
Up to a global scalar, there exists a unique modified trace $\mt$ on $\proj$.
\end{proposition}

For a simple object $V\in\proj$, let $\mRT_{V}$ denote the
\emph{renormalized Reshetikhin--Turaev invariant} of links whose components are
all colored by $V$. If $L$ is a framed link given by the closure of a $(1,1)$-tangle
$T$, then
\begin{align}
    \mRT_{V,L}=\la\RT_{V,T}\ra\cdot \md(V) = \mt(\RT_{V,T})
\end{align}
where $\RT_{V,T}=\la\RT_{V,T}\ra \cdot \mathrm{id}_V$ is the image of $T$ under
the usual RT functor. The scalar noted in Proposition \ref{prop.normalize}
determines the normalization of the unknot for $\mRT_V$.

Note that $\mRT$ is a well-defined invariant of links whose components may be
colored by different simple projective objects, but this generalization will not
be considered in this article.

\begin{lemma}
\label{lem.11}  
Let $T$ be a $(1,1)$-tangle with no closed components, i.e., a long knot.
If $V, W\in\cat$ and $f\in \Hom_{\cat}(V,W)$, then 
\be
\label{FWT}
\RT_{W,T}\circ f = f\circ \RT_{V,T}
\ee
\end{lemma}
The lemma is a consequence of the naturality of the braiding and duality morphisms.

\begin{proof}[Proof of Theorem~\ref{thm.n0cable}]
Let $\knot$ be a knot and $V\in\proj$ simple. As indicated in \cite[Prop.18]{GPT},
\be
\mRT_{V,\knot^{(0,n)}}=\mRT_{V^{\otimes n}, \knot}\,.
\ee
Fix a collection of maps $f_i:V_i\to V^{\otimes n}$ and $g_i:V^{\otimes n}\to V_i$
witnessing the isomorphism, i.e. $\sum_i f_i\circ g_i=\id_{V^{\otimes n}}$ and
$g_i\circ f_j=\delta_{ij}\, \id_{V_i}$. Let $\tangle$ be a $(1,1)$-tangle with closure
$\knot$. Then 
\be
\begin{aligned}
  \mRT_{V,\knot^{(n,0)}} &= \mRT_{V^{\otimes n},\knot} =
  \mt_{V^{\otimes n}}(\RT_{V^{\otimes n},
\tangle}) =
  \mt_{V^{\otimes n}}\left(\RT_{V^{\otimes n}, \tangle}\circ
\left(\sum_i f_i\circ g_i\right)\right)
\\
&=\sum_i \mt_{V^{\otimes n}}\left( f_i \circ \RT_{V_i, \tangle}\circ g_i \right)
=\sum_i \la \RT_{V_i, \tangle}\ra 
\mt_{V^{\otimes n}}\left(f_i\circ g_i\right)
\\
&=
\sum_i \la \RT_{V_i, \tangle}\ra \mt_{V_i}\left(g_i\circ f_i\right)
=
\sum_i \la \RT_{V_i, \tangle}\ra \cdot \mt_{V_i}\left(\id_{V_i}\right)
=
\sum_i \mRT_{V_i, \knot}
\end{aligned}
\ee
where the first equality in the second line uses Lemma~\ref{lem.11} and the
first equality in the third line uses cyclicity of the trace.
\end{proof}

%%%%%%%%%%%%%%%%%%%%%%%%%%%%%%%%%%%%%%%%%%%%%%%%%%%%%%%%%%%%%%%%%%%%%%%%%%%%
%%%%%%%%%%%%%%%%%%%%%%%%%%%%%%%%%%%%%%%%%%%%%%%%%%%%%%%%%%%%%%%%%%%%%%%%%%%%

\section{Representation theory for \texorpdfstring{$\UqGH$}{UqHsl(2|1)}}

We assume the ground field is the complex numbers $\BC$ unless stated otherwise.
Fix $q\in\BC^\times$ not a root of unity. For $x\in\BC$, let $\br{x}=q^x-q^{-x}$ and
$[x]=\frac{\br{x}}{\br{1}}$.

\subsection{Super vector spaces}

A \emph{super vector space} is a $\BZ_2$-graded vector space $V=V_{\p0}\oplus V_{\p1}$.
The parity of a homogeneous element $v\in V$ is denoted $\p v \in \BZ_2$. Morphisms
of super vector spaces are parity preserving linear maps. The tensor product of
super vector spaces $V$ and $W$ is the tensor product of the underlying vector spaces
with $\BZ_2$-grading given by
\be
\begin{aligned}
(V \otimes W)_{\p p} = \bigoplus_{\p a + \p b = \p p} V_{\p a} \otimes W_{\p b}.
\end{aligned}
\ee
A (left) module over a (associative) superalgebra $A$ is a super vector space $M$
with an $A$-module structure for which the action map $A \otimes M \rightarrow M$ is
a morphism of super vector spaces. Given homogeneous elements $a,b \in A$, set
$[a,b] = ab -(-1)^{\p a \p b} ba$.

\subsection{The unrolled quantum group}

In this section we describe the unrolled quantum group $\UqGH$, which is a version of
the usual quantized Lie superalgebra $\UqG$ that includes the logarithm of the
Cartan generators. This notion was coined in~\cite{CGP:remarks}, but the name
unrolled is not descriptive, and neither is the superscript in the notation $\UqGH$
that indicates that the logarithm of the Cartan generators is taken.

Aside from the terminology and notation, the presence of the additional generators
in $\UqGH$ versus $\UqG$ ensures that the category of weight modules has a
well-defined braiding and twisting; see Section~\ref{sub.braiding} below.

Recall the Cartan data $(a_{ij})=\begin{pmatrix} 2 &-1\\-1&0 \end{pmatrix}$ associated
to the Lie superalgebra $\sl21$. 

\begin{definition}
Let $\UqGH$ be the unital superalgebra with even generators $H_i$, $K_i$, $K_i^{-1}$,
$E_1$, $F_1$ for $i=1,2$ and odd generators $E_2$, $F_2$ and relations 
\be
\label{eq:rels}
\begin{aligned}
  [H_i,H_j]=[H_i,K_j^\pm]=0\,,\quad K_iK_i^{-1}=1\,,\quad
  [E_i,F_j]=\delta_{ij}\frac{K_i-K_i^{-1}}{q-q^{-1}}\,,
\quad E_2^2=F_2^2=0\,,
\\
[H_i,E_j]={a_{ij}} E_j\,,
\quad
[H_i,F_j]=-{a_{ij}} F_j\,,
\quad 
K_iE_j=q^{a_{ij}}E_jK_i\,,\quad K_iF_j=q^{-a_{ij}}F_jK_i\,,
\end{aligned} 
\ee
\begin{align}
E_1^2E_2-[2]E_1E_2E_1+E_2E_1^2=0\,,
\quad
F_1^2F_2-[2]F_1F_2F_1+F_2F_1^2=0\,.
\end{align}
\end{definition}

There is a unique Hopf superalgebra structure on $\UqGH$ with counit $\epsilon$,
coproduct $\Delta$\footnote{We hope that our notation for the coproduct is not confused
  with our notation for the Alexander polynomial.} and antipode $S$ defined
on generators by
\be
\label{eq:HopfStructure}
\begin{aligned}
\Delta(E_i)&=E_i\otimes K_i+1\otimes E_i\,,
&S(E_i)&=-E_iK_i^{-1}\,,
&\epsilon(E_i)&=0\,,
\\
\Delta(F_i)&=F_i\otimes 1+ K_i^{-1}\otimes F_i\,,
&S(F_i)&=-K_iF_i\,,
&\epsilon(E_i)&=0\,,
\\
\Delta(H_i)&=H_i\otimes 1+1\otimes H_i\,,
&S(H_i)&=-H_i\,,
&\epsilon(H_i)&=0\,,
\\
\Delta(K_i)&=K_i\otimes K_i\,,
&S(K_i)&=K_i^{-1}\,,
&\epsilon(K_i)&=1\,.
\end{aligned}
\ee

The Hopf superalgebra $\UqGH$ is the \emph{unrolled quantum group} of
$\sl21$, which below we will abbreviate simply by $\UqH$. The superscript $H$
refers to the fact that it contains generators $H_i$ corresponding to the Cartan
algebra. At any rate, $\UqH$ contains 
the Hopf super-subalgebra $\U_q$ generated by $E_i$, $F_i$, $K_i$ for $i=1,2$.
  
We define the odd elements
\be
\label{E12F12}
E_{12}=E_1E_2 - q E_2E_1, \qquad F_{12}=F_2F_1 - q^{-1} F_1F_2 \,.
\ee
Let $\U_q^+$ be the algebra generated by $E_1$ and $E_2$, which also contains
$E_{12}$, and has the following relations
\be
E_1E_{12}=q^{-1}E_{12}E_1\,, \qquad E_{12}E_2=-q E_2E_{12}\,,
\qquad E_1E_2=qE_2E_1+E_{12}\,,
\ee
and similar relations hold in $\U_q^-$ which is generated by
$F_1$ and $F_2$
\be
F_1F_{12}=q^{-1}F_{12}F_1\,, \qquad F_{12}F_2=-q F_2F_{12}\,,
\qquad F_1F_2=qF_2F_1-qF_{12}\,.
\ee
Further relations in $\Uq$ include
\be
\begin{aligned}
[E_1,F_{12}]&=-q^{-1}F_2K_1^{-1}\,,
&
[E_2,F_{12}]&=F_1K_2\,,
&[E_1,F_1^i]&=[i]F_1^{i-1}\frac{q^{-(i-1)}K_1-q^{i-1}K_1^{-1}}{q-q^{-1}}\,,
\\
[F_1,E_{12}]&=E_2K_1\,,
&
[F_2,E_{12}]&=qE_1K_2^{-1}\,,
&
[E_{12},F_{12}]&=\frac{K_1K_2-K_1^{-1}K_2^{-1}}{q-q^{-1}}\,.
\end{aligned}
\ee

There are two natural bases of $\U_q^+$ and $\U_q^-$. One is the PBW basis generated by
\be
\label{PBW}
E^\psi=E_2^{\psi(\alpha_2)}E_{12}^{\psi(\alpha_{12})}E_1^{\psi(\alpha_1)},
\qquad
F^\psi=F_2^{\psi(\alpha_2)}F_{12}^{\psi(\alpha_{12})}F_1^{\psi(\alpha_1)} \,,
\ee
using the ordered set of positive roots $\Phi^+=(\alpha_2,\alpha_{12}, \alpha_1)$
with $\psi:\Phi^+\to \BZ_{\geq0} \times\{0,1\} \times \{0,1\}$.
To further simplify notation we will write $E^{(a,b,i)}=E^\psi$ and $F^{(a,b,i)}=F^\psi$,
where $\psi(\a_2)=a$, $\psi(\a_{12})=b$ and $\psi(\a_1)=i$. 

The second basis, given here for $\U_q^-$, consists of the vectors $F_{i,j}$ where
\be
\label{nonPBW}
F_{i,0}=F_1^i \,, \qquad
F_{i,1}=F_2F_1^i \,, \qquad
F_{i,2}=F_1F_2F_1^i \,, \qquad
F_{i,3}=F_2F_1F_2F_1^i
\ee
for $i \in \BZ_{\geq0}$, and likewise for $\U_q^+$. 

For  $\b\in\BZ^2$ and $\g\in\BZ_{\geq0}^2$ write $K^\b=K_1^{\b_1}K_2^{\b_2}$ and
$H^\g=H_1^{\g_1}H_2^{\g_2}$. PBW bases for $\UqH$ and $\Uq$ consist of the vectors
\be
F^{\psi'}K^\b H^\g E^{\psi'} \qquad \mbox{and} \qquad F^{\psi'}K^\b E^{\psi'},
\ee
respectively, where $\psi,\psi'$ are as in Equation \eqref{PBW}, $\b\in\BZ^2$,
and $\g\in\BZ_{\geq0}^2$.

\subsection{The category \texorpdfstring{$\catsl$}{sl(2|1)-mod}}
\label{sub.catsl21}

Let $\mathfrak{h}\subset\UqH$ denote the Cartan subalgebra of $\sl21$ generated
by $H_1$ and $H_2$. A $\UqGH$-\emph{weight module} is a finite-dimensional
$\UqGH$-module that satisfies
\be
H_i v = \lambda_i v\,, \qquad K_i v = q^{\lambda_i} v \qquad \mbox{for $i=1,2$}
\ee
for every vector $v \in V$ of weight $\lambda=(\lambda_1,\lambda_2)$.
Let $\catsl$ denote the category whose objects are $\UqGH$-weight modules and morphisms
being $\UqGH$-linear maps. It is a $\BC$-linear abelian category with monoidal
structure being the tensor product of $\UqGH$-weight modules. It also has a braiding
and a ribbon structure induced by the ribbon Hopf super-algebra structure of $\UqGH$, discussed below. 
The duality maps on $\catsl$ are given by
\be
\begin{aligned}
\lev_V(f \otimes v)&= f(v)\,, 
& \quad
\lcoev_V(1) &= \sum_i v_i \otimes v_i^{*}\,, 
\\
\rev_V(v \otimes f) &= 
(-1)^{\p f \p v}f(K_2^{-2} v)\,,
& \quad
\rcoev_V(1)&= \sum_i (-1)^{\p v_i}v_i^{*} \otimes K_2^{2} v_i
\end{aligned}
\ee
for any basis $\{v_i\}$ of $V\in\catsl$ and associated dual basis $\{v_i^*\}$ of
$V^*$. Since $\catsl$ is a category of modules over a ring, this category has
enough projectives. Consequently, we have an ideal $\calP^{\sl21}$ of projectives
in $\catsl$ which is a full ribbon subcategory of $\catsl$.

We consider a family of indecomposable $\UqGH$-weight modules
$V(n,\a)$ which are determined by their highest weight~\cite{Kac:superalgebra}
\be
\label{Lambda}
(n,\a) \in \Lambda = \BZ_{\geq 0} \times \BC
\ee
consisting of one discrete variable $n$ and one continuous variable $\a$ coming
from the actions of $H_1$ and $H_2$, respectively.
The dimension of $V(n,\a)$ is $4(n+1)$ and a PBW basis
for the module $V_p(n,\alpha)$ with highest weight vector $v_0$ is given in terms
of the PBW basis for $\U_q^-$:
\be 
\begin{aligned}
\{v_{abi}=F^{abi}v_0\mid i\in\{0,\ldots, n\}, a,b\in\{0,1\}\}\,. 
\end{aligned}
\ee
Another basis for $V_p(n,\alpha)$ consists of the vectors
\be
\label{eq.nonPBW}
\{
v_{i,j}=F_{i,j}v_0\mid i\in\{0,\dots, n\}, j\in\{0,1,2,3\}
\}
\ee
which will be more convenient for certain computations.
The action of simple root vector generators on these basis vectors is given by 
\be
\begin{aligned}
F_1v_{i,k}&=
\begin{cases}
v_{i+1,0} &k=0\\
v_{i,2} &k=1\\
[2]v_{i+1,2}-v_{i+2,1} &k=2\\
v_{i+1,3} &k=3
\end{cases} 
\,,
&
E_1v_{i,k}&=
\begin{cases}
[i][n-i+1]v_{i-1,0} &k=0\\
[i][n-i+1]v_{i-1,1} &k=1\\
[n-2i+1]v_{i,1}+[i][n-i+1]v_{i-1,2} &k=2\\
[i][n-i+1]v_{i-1,3} &k=3
\end{cases} 
\,,
\\
F_2v_{i,k}&=
\begin{cases}
v_{i,1} &k=0\\
v_{i,3} &k=2\\
0 &k=1,3
\end{cases} 
\,,
&
E_2v_{i,k}&=
\begin{cases}
0 &k=0\\
[\a+i]v_{i,0} &k=1\\
[\a+i]v_{i+1,0} &k=2\\
[\a+i+1]v_{i,2}-[\a+i]v_{i+1,1} &k=3
\end{cases} 
\end{aligned}
\ee
where $v_{i,k}=0$ if $i<0$ or $i>n$. The Cartan generators act diagonally according to the formulas
\be
\begin{aligned}
    H_1v_{i,k}
&=
\begin{cases}
    (n-2i+k)v_{i,k} & k=0,1\\
    (n-2(i+1)+k-1)v_{i,k} & k=2,3
\end{cases}\,,
&&
H_2v_{i,k}
&=
\begin{cases}
    (\a+i)v_{i,k} & k=0,1\\
    (\a+i+1)v_{i,k} & k=2,3
\end{cases}
\end{aligned}\,.
\ee

Another point in the discussion of these modules, is that their highest weight
vector $v_0$ can be even (in which case we denote them by $V_+(n,\a)$), or odd
(in which case we denote them by $V_-(n,\a)$). The parity data was not considered
explicitly
in \cite{Geer:multi}, but it is important when considering the braiding on summands of
tensor powers of $V_{\pm}(0,\alpha)$. In particular, we use the braiding to compute the
modified dimensions appearing in Equation \eqref{eq:ModifiedDimension} and
Theorem \ref{thm.n0}. Below, we will denote $V_+(n,\alpha)$ simply by $V(n,\a)$.

\subsection{Simple versus projective modules}

Here we discuss two flavors of the modules $V(n,\a)$, typical and atypical. The
dichotomy between typical and atypical modules is reflected in their irreducibility
and projectivity properties. Since the grading of the highest weight vector in
$V(n,\a)$ is only relevant when discussing braiding, all results of this section
apply to both $V_\pm(n,\a)$.

Following \cite{Kac:typical}, we say that $V(n,\a)$ or the
weight $(n,\a)$ is \emph{typical} if $\a(n+1+\a) \neq 0$ and is
\emph{atypical} otherwise.

\begin{remark}
\label{rem.Vdual}
Note the isomorphism 
\be
\label{VVdual}
V_{\pm}(n,\a)^*\cong V_{\pm}(n,-\a-n-1)\ee
which follows from inspecting the weights
of $V_{\pm}(n,\a)$ and preserves typicality.
\end{remark}

\begin{remark}
\label{rem.P}
An interesting example of a projective weight
module, which is not typical but appears in the
tensor product decomposition of two typical weight modules, is the 8-dimensional
module $P$ from Figure \ref{fig.proj} below.
\end{remark}

\begin{lemma}
The module $V(n,\a)$ is irreducible if and only if it is typical.
\end{lemma}

\begin{proof}
The representation is irreducible if and only if the highest weight vector space
$\la v_0\ra$ is contained in the submodule generated by the lowest weight vector
$F^{11n}v_0$. By the PBW theorem, an algebra element which lifts $F^{11n}v_0$ to
$v_0$ is a multiple of $E^{11n}$. A computation now shows
\be
E^{11n}F^{11n}v_0=[\a][n+\a+1]([n]!)^2 v_0
\ee
which is nonzero if and only if $\a(n+1+\a) \neq 0$.
\end{proof}

\begin{lemma}
\label{lem.proj}  
The representation $V(n,\a)$ is projective if and only if it is typical.
\end{lemma}

\begin{proof}
Let $f: W \twoheadrightarrow Z$ be an epimorphism in $\cat$. A nonzero morphism
$g: V(n,\a)\rightarrow Z$ is determined by a highest weight vector $g(v_0)=z_0 \in Z$
of weight $(n,\a)$. We aim to prove that there is a lift $g':V(n,\a)\to W$
such that $g=f\circ g'$.

Let $w_0 \in W$ be a preimage of $z_0$
through $f$ which is necessarily nonzero. Moreover, for $i=1,2$
\be
f(E_iw_0)=E_if(w_0)=E_iz_0=E_ig(v_0)=g(E_iv_0)=0\,.
\ee
Therefore, a weight vector $w_0'$ in the submodule generated by $w_0$ satisfying
$f(w_0')=z_0$ must be of the following form, in the notation of Equation \eqref{PBW},
\be
\begin{aligned}
w_0'=\sum_{\psi,\psi'\in \BZ_{\geq0}^{\Phi^+}} c_{\psi,\psi'} F^\psi E^{\psi'} w_0
\end{aligned}
\ee
with $c_{(000)(000)}=1$ and only finitely many $c_{\psi,\psi'}\in\BC$ nonzero.
By weight considerations we have the following possible combinations of $\psi$
and $\psi'$:
\be
\begin{aligned}
\psi&=\psi'=(a,b,i), \\
\psi &= (1,0,i+1), \psi'=(0,1,i),\\
\psi &= (0,1,i),\psi'=(1,0,i+1)\,.
\end{aligned}
\ee
Let $k$ be the largest $i$ such that $E_1^iw_0$ is nonzero. 

We wish to determine $c_{\psi,\psi'}$ so that $w_0'$ is a highest weight vector in $W$,
i.e. $E_1w_0'=E_2w_0'=0$. Following weight considerations, it is necessary that
$F_1^{n+1}w_0=0$. To impose the highest weight constraint, we compute these actions
of $E_1$ and $E_2$ and find a relation among the coefficients $c_{\psi,\psi'}$. We
consider six cases each. For $E_1$:
\be
\begin{aligned}
  E_1\cdot F^{00i}E^{00i}w_0&
  = F^{00i}E^{00i+1}w_0+[i][n+i+1]F^{00i-1}E^{00i}w_0
\\
E_1\cdot F^{01i}E^{01i}w_0&
=q^{-1}F^{01i}E^{01i+1}w_0+[i][n+i+2]F^{01i-1}E^{01i}w_0-q^{-n-2}F^{10i}E^{01i}w_0
\\
E_1\cdot F^{10i}E^{10i}w_0&
=qF^{10i}E^{10i+1}w_0+F^{10i}E^{01i}w_0+[i][n+i]F^{10i-1}E^{10i}w_0
\\
E_1\cdot F^{11i}E^{11i}w_0&=F^{11i}E^{11i+1}w_0+
[i][n+i+1]F^{11i-1}E^{11i}w_0
\\
E_1\cdot F^{10i+1}E^{01i}w_0&=
q^{-1}F^{10i+1}E^{01i+1}w_0+[i+1][n+i+1]F^{10i}E^{01i}w_0
\\
E_1\cdot F^{01i}E^{10i+1}w_0&=qF^{01i}E^{10i+2}w_0+
F^{01i}E^{01i+1}w_0
\\&+[i][n+i+2]F^{01i-1}E^{10i+1}w_0
-q^{-n-2}F^{10i}E^{10i+1}w_0
\end{aligned}
\ee
and for $E_2$:
\be
\begin{aligned}
E_2\cdot F^{00i}E^{00i}w_0&=F^{00i}E^{10i}w_0
\\
E_2\cdot F^{01i}E^{01i}w_0&=-F^{01i}E^{11i}w_0+q^{\a-1}F^{00i+1}E^{01i}w_0
\\
E_2\cdot F^{10i}E^{10i}w_0&=[\a]F^{00i}E^{10i}w_0
\\
E_2\cdot F^{11i}E^{11i}w_0&=-q^{\a-1}F^{10i+1}E^{11i}w_0+[\a]F^{01i}E^{11i}w_0
\\
E_2\cdot F^{10i+1}E^{01i}w_0&=-F^{10i+1}E^{11i}w_0+[\a]F^{00i+1}E^{01i}w_0
\\
E_2\cdot F^{01i}E^{10i+1}w_0&=q^{\a-1}F^{00i+1}E^{10i+1}w_0\,.
\end{aligned}
\ee
These determine a system of equations in the coefficients from $E_1$:
\be
\begin{aligned}
  c_{(00i)(00i)}+[i+1][n+i+2]c_{(00i+1)(00i+1)}&=0
  % \quad (00i),(00i+1)
\\
q^{-1}c_{(01i)(01i)}+{[i+1][n+i+3]}c_{(01i+1)(01i+1)}+c_{(01i)(10i+1)}&=0
% \quad (01i),(01i+1)
\\
-q^{-n-2}c_{(01i)(01i)}+c_{(10i)(10i)}+q^{-1}c_{(10i)(01i-1)}
+{[i+1][n+i+1]}c_{(10i+1)(01i)}&=0%\quad (10i),(01i)
\\
qc_{(10i)(10i)} +{[i+1][n+i+1]}c_{(10i+1)(10i+1)}-q^{-n-2}c_{(01i)(10i+1)}&=0
% \quad (10i),(10i+1)
\\
c_{(11i)(11i)}+{[i+1][n+i+2]}c_{(11i+1)(11i+1)}&=0 %\quad(11i),(11i+1)
\\
qc_{(01i)(10i+1)}+{[i+1][n+i+3]}c_{(01i+1)(10i+2)}&=0% \quad (01i),(10i+2)
\end{aligned}
\ee
and from $E_2$:
\be
\begin{aligned}
c_{(00i)(00i)}+{[\a]}c_{(10i)(10i)}+q^{\a-1}c_{(01i-1)(10i)}&=0%\quad (00i)(10i)
\\
-c_{(01i)(01i)}+{[\a]}c_{(11i)(11i)}&=0%\quad (01i),(11i)
\\
q^{\a-1}c_{(01i)(01i)}+{[\a]}c_{(10i+1)(01i)}&=0% \quad (00i+1),(01i)
\\
-q^{\a-1}c_{(11i)(11i)}-c_{(10i+1)(01i)}&=0 %\quad (10i+1)(11i)
\end{aligned}
\ee
Since $i\leq k$, we solve the finite linear system which has unique solution in
the parameters $c_{\psi,\psi'}$ if and only if $\a(n+\a+1)\neq 0$: 
\be
\begin{aligned}
c_{(100)(100)}&=-\frac{1}{[\a]}\,,&
c_{(110)(110)}&=-\frac{q}{[n+\a+1][\a]}\,,\\
c_{(00i)(00i)}&=\frac{(-1)^i[n+1]!}{[i]![n+i+1]!}\,,&
c_{(11i)(11i)}&=\frac{(-1)^{i+1}q[n+1]!}{[i]![n+i+1]![n+\a+1][\a]}\,,\\
c_{(01i)(10i+1)}&=\frac{(-1)^{i+1}q^{n+i+3}[n+1]!}{[i]![n+i+2]![n+\a+1]}\,,&
c_{(01i)(01i)}&=\frac{(-1)^{i+1}q[n+1]!}{[i]![n+i+1]![n+\a+1]}\,,\\
c_{(010)(101)}&=\frac{-q^{n+3}}{[n+2][n+\a+1]}\,,&
c_{(10i+1)(01i)}&=\frac{(-1)^{i}q^{\a}[n+1]!}{[i]![n+i+1]![n+\a+1][\a]}\,,
\end{aligned}
\ee
\begin{align}
c_{(10i+1)(10i+1)}&=(-1)^iq^{i+1}\frac{[n+\a+i+2][n+1]!}{[i+1]![n+i+2]![n+\a+1][\a]}\,.
\end{align}
% \be
% \begin{aligned}
% c_{(00i)(00i)}&=(-1)^i\frac{[n+1]!}{[i]![n+i+1]!},&
% c_{(11i)(11i)}&=(-1)^i\frac{[n+1]!}{[i]![n+i+1]!}c_{(110)(110)},\\
% c_{(01i)(10i+1)}&=(-q)^{i}\frac{[n+2]!}{[i]![n+i+2]!}c_{(010)(101)},&
% c_{(01i)(01i)}&=(-1)^i\frac{[\a][n+1]!}{[i]![n+i+1]!}c_{(110)(110)},\\
% c_{(010)(101)}&=q^{n+2}\frac{[\a]}{[n+2]}c_{(110),(110)},&
% c_{(10i+1)(01i)}&=(-1)^{i+1}q^{\a-1}\frac{[n+1]!}{[i]![n+i+1]!}c_{(110)(110)}\\
% c_{(100)(100)}&=-\frac{1}{[\a]}&
% c_{(110)(110)}&=-\frac{q}{[n+\a+1][\a]}\quad (100),(010)\\
% c_{(10i+1)(10i+1)}&=(-1)^iq^{i+1}\frac{[n+\a+i+2][n+1]!}{[i+1]![n+i+2]![n+\a+1][\a]}
% \end{aligned}
% \ee
In which case,
the map $V(n,\a)\to W$ determined by $v_0\mapsto w_0'$ is a
lift of $g$.
\end{proof}

\begin{remark}
\label{rem.avoid}   
Although there exists a unique (up to scalar) modified trace on $\calP^{\sl21}$, we
are mostly interested in the typical modules, their modified dimensions, and their
tensor product decompositions. For instance, our main Theorem~\ref{thm.n0}
and Conjecture~\ref{conj.1} involve typical modules. 
\end{remark}

\subsection{Fusion rules}
\label{sub.fusion}

In this section we recall the tensor product decomposition of two typical modules
using a formula of Geer-Patureau-Mirand~\cite[Lem.1.3]{Geer:multi}
\be
\label{V0n}
\begin{aligned}
V(0,\a) \otimes V(n,\beta) \cong~ &  V(n,\a+\beta) \oplus V(n,\a+\beta+1) \oplus
V(n+1,\a+\beta) \\ & \oplus (1-\delta_{n,0}) V(n-1,\a+\beta+1) 
\end{aligned}
\ee
with the understanding that the last term is omitted when $n=0$. Here, we assume that
$\a,\beta, 1$ are $\BZ$-linearly independent, hence the corresponding representations
are typical. We complement the above decomposition with highest weight vectors.

\begin{lemma}
\label{lem.high}  
A choice of highest weight vectors of the summands in the decomposition
~\eqref{V0n} $V(0,\a) \otimes V(n,\beta)$ of are:
\be
\label{highv}
\begin{aligned}
V(n,\a+\beta) & : \,\, \Delta(E^{(110)}F^{(110)})(v_0\otimes v_0), \\
V(n,\a+\beta+1) & : \,\, \Delta(E^{(110)}F^{(110)})(v_0\otimes F_2F_1F_2v_0), \\
V(n+1,\a+\beta) & : \,\, \Delta(E^{(110)}F^{(110)})(v_0\otimes F_2v_0), \\
V(n-1,\a+\beta) & : \,\,
\Delta(E^{(110)}F^{(110)})(v_0\otimes ([n]F_1F_2-[n+1]F_2F_1)v_0)\,.
\end{aligned}
\ee
\end{lemma}

\begin{proof}
These vectors are clearly highest weight and a straightforward computation shows
that they are nonzero when $\a$, $\b$, and $1$ are linearly independent over $\BZ$.
\end{proof}

\begin{remark}
\label{rem.parity}  
The parity of the above highest weight vectors implies the refinement
\be
\label{Vp0n}
\begin{aligned}
  V_p(0,\a) \otimes V_{p'}(n,\beta) \cong &~
  V_{p+p'}(n,\a+\beta) \oplus V_{p+p'}(n,\a+\beta+1) \oplus
V_{p+p'+1}(n+1,\a+\beta) \\ & \oplus (1-\delta_{n,0}) V_{p+p'+1}(n-1,\a+\beta+1)\,. 
\end{aligned}
\ee
In other words, the parity of a summand with highest weight $(n+k,\alpha+\beta+\ell)$
is $\p{p+p'+k}$. Thus, in the formulas below, we can ignore the parity of the
representations. 
\end{remark}

This allows us to inductively decompose $V(0,\a)^{\otimes n}$ into
irreducible representations. A nice formula for such a decomposition was found by
Anghel~\cite{Anghel} who proved that for all integers $n \geq 1$ we have
\be
\label{texp}
V(0,\a)^{\otimes n} \cong \bigoplus_{k + \ell \leq n-1} m^{(n)}_{k,\ell} V(k,n\a+\ell)
\ee
where the summation is over all pairs $(k,\ell)$ of nonnegative integers
$k,\ell \geq 0$ satisfying $k+\ell \leq n-1$. If we replace $V(0,\alpha)$ with
$V_p(0,\alpha)$, then the summand $V(k,n\alpha+\ell)$ is replaced by
$V_{np+k}(k,n\alpha+\ell)$. From this, we can infer the parity in the lines below.
A bijective correspondence of paths
whose number is $m^{(n)}_{k,\ell}$ is given in the above mentioned work.
The numbers $m^{(n)}_{k,\ell}$ satisfy the recursion~\eqref{mrec}
with $m^{(n)}_{n-1,0}=1$ and the understanding that $m^{(n)}_{k,\ell}=0$ if
$k<0$ or $\ell<0$. They can be arranged in a triangular pattern similar to
the binomial coefficients, from which follows that they satisfy 
the symmetry
\be
\label{msym}
m^{(n)}_{k,\ell}=m^{(n)}_{k,n-1-k-\ell} \,.
\ee
The values for $n=1,\dots,6$ are given by

\be
\left[\begin{matrix}
  1
\end{matrix}\right]
\quad
\left[\begin{matrix}
  1 \,\, 1 \\ 1
\end{matrix}\right]
\quad
\left[\begin{matrix}
  1 \,\, 3 \,\, 1 \\ 2 \,\, 2 \\ 1
\end{matrix}\right]
\quad
\left[\begin{matrix}
1 \,\, 6 \,\, 6 \,\, 1 \\
3 \,\, 8 \,\, 3 \\ 3 \,\, 3 \\ 1
\end{matrix}\right]
\quad
\left[\begin{matrix}
1 \,\, 10 \,\, 20 \,\, 10 \,\, 1 \\
4 \,\, 20 \,\, 20 \,\, 4 \\
6 \,\, 15 \,\, 6 \\ 4 \,\, 4 \\ 1
\end{matrix}\right]
\quad
\left[\begin{matrix}
1 \,\, 15 \,\, 50 \,\, 50 \,\, 15 \,\, 1 \\
5 \,\, 40 \,\, 75 \,\, 40 \,\, 5 \\
10 \,\, 24 \,\, 10 \\ 5 \,\, 5 \\ 1
\end{matrix}\right] \,.
\ee
For example, we have
\be
\label{V2a}
\begin{aligned}
V(0,\a)^{\otimes 2} &\cong V(0,2\a) \oplus V(0,2\a+1) \\
& \hspace{1.5cm} \oplus V(1,2\a)
\end{aligned}
\ee
and
\be
\label{V3a}
\begin{aligned}
V(0,\a)^{\otimes 3} &\cong V(0,3\a) \oplus 3 V(0,3\a+1) \oplus V(0,3\a+2) \\
& \hspace{1.5cm} \oplus 2 V(1,3\a)  \oplus 2 V(1,3\a+1) \\
& \hspace{2.5cm} \oplus V(2,3\a) 
\end{aligned}
\ee

Equation~\eqref{texp} implies that every representation is a virtual sum
of tensor powers of $V_p(0,\a)$. A lesser-known property is the following.

\begin{lemma}
\label{lem.tp}  
Equation~\eqref{texp} determines the tensor product
decomposition of $V(n,\a) \otimes V(m,\b)$ uniquely. Explicitly, 
for $n, m \geq 0$ and $\a,\b,1$ $\BZ$-linearly independent, we have
\be
\label{Vnm}
\begin{aligned}
V(n,\a) \otimes V(m,\beta) \cong&~ 
V(n+m+1,\gamma) \\
  &\oplus \bigoplus_{k=|n-m|}^{n+m} ( V(k, \gamma+\mu+
\lfloor \tfrac{2+|n-m|-k}{2} \rfloor) \oplus
V(k, \gamma+\mu + \lfloor \tfrac{1+|n-m|-k}{2} \rfloor) \\  &\oplus 
(1-\delta_{n,m}) \,\, V(|n-m|-1, \gamma+1+\mu) 
\end{aligned}
\ee
where $\mu=\min\{n,m\}$ and $\gamma=\a+\beta$.
%with the understanding that $V(-1,\cdot)=0$. 
\end{lemma}

Note that the decomposition~\eqref{Vnm} has multiplicities 1 and 2 where multiplicity
2 appears in the middle sum when $k=1 \bmod |n-m|$. For example, we have:
\be
\label{V1V1}
\begin{aligned}
  V(1,\a) \otimes V(1,\beta)  \cong~&
  V(0,\gamma+1)\oplus V(0,\gamma+2) \\ &
  \oplus 2 V(1,\gamma+1) \oplus V(2,\gamma)
  \oplus V(2,\gamma+1) \\ & \oplus V(3,\gamma)
\end{aligned}
\ee
and
\be
\label{V2V2}
\begin{aligned}
  V(2,\a) \otimes V(2,\beta) \cong~& 
  V(0,\gamma+2)+V(0,\gamma+3) \\ &
  \oplus 2 V(1,\gamma+2) \oplus V(2,\gamma+1) \oplus V(2,\gamma+2) \\ &
  \oplus 2V(3,\gamma+1) \oplus V(4,\gamma) \oplus V(4,\gamma+1) \\ & 
  \oplus V(5,\gamma) 
\end{aligned}
\ee
with $\gamma=\a+\beta$.

\begin{proof}  
A change of variables of ~\cite[Sec.1.1,p.285]{Geer:multi} implies that
the character of a typical module $V(n,\a+r)$ for $n \in \BZ_{\geq 0}$, $\a \in \BC$
and $r \in \BZ$ is given by
\be
\label{char.Vna}
\mathrm{ch}(V(n,\a+r)) = t^\a (xy)^r s_n(x,y), \qquad
s_n(x,y) = \sum_{i=0}^{n} x^i y^{n-1} = y^n \frac{1-(x/y)^{n+1}}{1-x/y} \,.
\ee
Since a character uniquely characterizes typical modules, this and a combinatorial
identity (easily proven by induction on $n,m$) implies Equation~\eqref{Vnm}.
\end{proof}

\begin{remark}
With effort, one can supply highest weight vectors in the decomposition~\eqref{Vnm}
generalizing Lemma~\ref{lem.high}. 
\end{remark}

\begin{remark}
\label{rem.mutation}  
The multiplicity free decomposition of $V(0,\a)^{\otimes 2}$ explains by TQFT
axioms why $\LG^{(1)}$ cannot detect Conway mutation, whereas
the multiplicity 2 decomposition of $V(n,\a)^{\otimes 2}$ for $n \geq 2$ explains
why $\LG^{(n)}$ {is able to} detect Conway mutation.
\end{remark}

\subsection{Braiding}
\label{sub.braiding}

In this section we describe the braiding in $\catsl$. 
Recalling the work of Khoroshkin, Tolstoy and Yamane \cite{KhoroshkinTolstoy,Yamane}
$\Uq$ has an explicit universal quasi-$R$-matrix. Recall $E_{12}$ and $F_{12}$
from Equation~\eqref{E12F12}. We then define
\be
\label{3R}
\begin{aligned}
\Check R_1 &= \sum_{i=0}^{\infty} \frac{q^{\tfrac{i(i-1)}{2}}}{[i]!}
(q-q^{-1})^{i}E_1^i \otimes F_1^i
\\
\Check R_2 &= 1\otimes 1 - (q-q^{-1})E_2 \otimes F_2 
\\
\Check R_{12} &= 1\otimes 1 - (q-q^{-1})E_{12} \otimes F_{12} 
\end{aligned}
\ee
where $\Check R_1$ belongs to a tensor product completion of $\Uq\otimes \Uq$.
The quasi-$R$-matrix is then given by $\Check R=\Check R_{2}\Check R_{12}\Check R_{1}$.

For any $V, W \in \catsl$, the action of $\Check R$ on $V\otimes W$ is well-defined
and we denote this action by the same name $\Check R\in \End_{\BC}(V \otimes W)$.
Next, define
\be
\label{Upsilon}
\Upsilon_{V,W} \in \End_{\BC}(V \otimes W), \qquad
\Upsilon_{V,W} (v \otimes w) = q^{\sum(a^{-1})_{ij}\lambda_i \mu_j} v \otimes w
= q^{-(\lambda_1 \mu_2+\lambda_2 \mu_1+2\lambda_2 \mu_2)} v \otimes w,
\ee
where $v\in V$ and $w \in W$ are of weight $\lambda=(\lambda_1,\lambda_2)$ and
$\mu=(\mu_1,\mu_2)$, respectively and $(a_{ij})$ is the Cartan matrix.
The above formula involves bilinear terms in $\mu$ and $\lambda$ and as such,
it is not obtained by an action of $K$, and instead
requires the action of $H_i$ in order to be well-defined. This requires
(and motivates) the introduction of the unrolled version $\UqGH$ of the
quantum group $\UqG$ as was pointed out in the very first page
of ~\cite{Geer:superalgebra}.
We now define the graded swap $\tau_{V,W}$, the braiding $c_{V,W}$
and the ribbon map $\th_V$ by 
\be
\label{cth}
\begin{aligned}
\tau_{V,W}: & \,\, V\otimes W\to W\otimes V, \qquad
\tau_{V,W}(v\otimes w)=(-1)^{\p v \p w}w\otimes v \\
c_{V,W} : & \,\, V \otimes W \to W \otimes V, \qquad 
c_{V,W}=\tau_{V,W}\circ \Upsilon_{V,W}\circ \Check R, \\
\th_V : & \,\, V \to V, \hspace{2.8cm} \th_V=\trr(c_{V,V}) \,.
\end{aligned}
\ee

\begin{proposition}[\cite{KhoroshkinTolstoy,Yamane}]
The family of maps $c=\{c_{V,W} : V \otimes W \to W \otimes V\}_{V,W \in \catsl}$
defines a braiding on $\catsl$.
\end{proposition}

The endomorphism $c_{V,V} \in \End_\cat(V \otimes V)$ is the $R$-matrix of the
representation $V$ and can be described explicitly using an ordered basis of $V$.
Using the ordered basis $(v_{0,0}, v_{0,1},v_{0,2}, v_{0,3})$ of $V=V_0(0,\a)$
from Equation~\eqref{eq.nonPBW}, and 
abbreviating the vectors $v_{0,i}\otimes v_{0,j}$ by $x_{i,j}$ for $i,j=0,\dots,3$
it follows that the matrix $\mathsf{R}_{\LG}=(c_{V,V}(x_{i,j}))_{0\leq i,j\leq 3}$ is
given by 
\begin{center}
\small
\resizebox{\textwidth}{!}{
$
  \mathsf{R}_{\LG}=q^{-2\a(\a+1)}\left(\begin{array}{@{}cccc}
  %i=1
  q^{2\a} x_{00}
  &
  q^{\a} x_{10}
  &
  q^{\a} x_{20}
  &
  x_{30}
  \\ %i=2
  q^{\a} (x_{01} + \br \a x_{10})
  &
  - x_{11}
  &
  -q^{-1} x_{21}+
  \br \a x_{30}
  &
  q^{-\a-1} x_{31}
  \\ %i=3
  q^{\a} (x_{02} + \br \a x_{20})
  &
  -q^{-1} x_{12}-{q^{-1}\br 1} x_{21}
  -
  q^{-1}\br \a x_{30}
  &
  - x_{22}
  &
  q^{-\a-1} x_{32}
  \\ %i=4
  x_{03}+q^{-1}\br{\a+1} (x_{12} -q^{-1}x_{21} + \br{\a} x_{30})
  &
  q^{-\a-1} (x_{13}-\br{\a+1} x_{31})
  &
  q^{-\a-1} (x_{23}-\br{\a+1} x_{32})
  &
  q^{-2\a-2} x_{33}
  \end{array}\right)
$}\,.
\end{center}

\begin{remark}
The above $R$-matrix does not agree with the one given in~\cite{GHKKST}. Denote
the $R$-matrix therein by $\mathsf{R}_{\text{skein}}$ which acts on a tensor product
space $W^{\otimes 2}$ where $W$ has a basis $w_1,\dots, w_4$. Although these
$R$-matrices are not equal, they are conjugate in the sense of
\cite[Sec.2.5]{GHKST} by the map $\vphi=f\circ \vphi'$ where 
\begin{align}
    f\in \End_\BC(V)\,, \qquad f(t_0)= (t_0)^{-1}={q^{-2\a}}\,, \qquad
f(t_1)= (t_1)^{-1}={q^{2(\a+1)}}
\end{align}
enacts the change of variables and $\vphi':W\to V$ is defined by
\begin{small}
\be
\label{eq.basis}
\vphi(w_1)= \sqrt{(t_0-1)(1-t_1)} v_{0,0}\,,
\quad
\vphi(w_2)= v_{0,2}\,,\quad
\vphi(w_3)= v_{0,1}\,,\quad
\vphi(w_4)= {\frac{1}{t_0^{1/2}-t_0^{-1/2}}} v_{0,3}\,.
\ee
\end{small}
\end{remark}

\subsection{Ribbon Structure}
\label{sub.ribbon}

In this section we prove that the map~\eqref{cth} defines a ribbon structure
in $\catsl$, using \cite[Thm.2]{Geer:trace} adapted to the case of the
category $\catsl$. We recall some notions related to generic semisimplicty in
order to invoke \cite[Thm.2]{Geer:trace}, which allows us to prove that the
natural transformation $\th$ from Equation \eqref{cth} defines a ribbon
structure on $\catsl$.

Recall the notions of a simple object of a $\BK$-linear category and semisimplicity from
Section~\ref{sub.proj}.

Fix an abelian group $\sG$. A pivotal $\BK$-linear category $\cat$ is
\emph{$\sG$-graded} if for each $g\in\sG$ we have a nonempty full subcategory
$\cat_g\subseteq \cat$ stable
under retract such that:
\begin{enumerate}
\item $\cat = \bigoplus_{g\in \sG}\cat_g$\,,
\item if $V\in\cat_g$, then $V^*\in\cat_{-g}$\,,
\item if $V\in \cat_g$ and $V'\in \cat_{g'}$, then $V\otimes V'\in \cat_{g+g'}$\,,
\item if $V\in \cat_g$, $V'\in \cat_{g'}$, and $\Hom_\cat(V,V')\neq 0$, then $g=g'$.
\end{enumerate}

A subset $\sX \subset \sG$ is \emph{symmetric} if $\sX=-\sX$ and \emph{small}
 if $\bigcup_{i=1}^n (g_i+\sX) \neq \sG$ for any finite collection
 $g_1,\ldots ,g_n\in \sG$.
 
A $\sG$-graded category $\cat$ is \emph{generically semisimple} if there
 exists a small symmetric subset $\sX \subset \sG$ such that each subcategory
 $\cat_g$ is semisimple for $g \in \sG \setminus \sX$.

It turns out that the category $\catsl$ is $\BC/\BZ$-graded in the sense of the
above definition. Indeed, given $\bar\a \in \BC/\BZ$, consider
the full subcategory $\catsl_{\bar{\a}}$ of $\catsl$ of modules whose
highest weights belong to the set $ \BZ_{\geq0} \times (\a+\BZ)$.
The tensor product given in Lemma~\ref{lem.tp} implies property
(3) above whereas (2) and (4) are clear. 

Set $\sX=\{\bar{0}\} \subset \BC/\BZ$ corresponding to highest weights in
$\BZ_{\geq0} \times \BZ$, that contain the weight of the identity
$\unit\in\catsl_{\bar{0}}$, as well as the set
$\{(n,0),(n,-n-1)\mid n\in\BZ_{\geq 0}\}$ of atypical weights. Then,
$\catsl$ is generically semisimple with respect to $\sX$.

\begin{proposition}
The family of maps $\th=\{\th_{V} : V \to V\}_{V \in \catsl}$ defines a ribbon
structure on $\catsl$.
\end{proposition}

\begin{proof}    
Since $\catsl$ is generically semisimple,
we may apply \cite[Thm.2]{Geer:trace}.
Therefore, it is sufficient to prove that $\th_V^*=\th_{V^*}$ for any typical
object $V\in\catsl$. Fix a typical object $V=V_p(n,\a)$. We compute the
composition 
\be
\begin{aligned}
  (\id_{V} \otimes \rev_{V})\circ (c_{V,V}\otimes \id_{V^*})
  \circ(\id_V\otimes \lcoev_{V})\,.
\end{aligned}
\ee
Let $v_0,v\in V$ with $v_0$ a highest weight vector. Then 
\be
\begin{aligned}
  c_{V,V}(v_0,v)=\tau_{V,V}\circ \Upsilon_{V,V}(v_0,v)
  =(-1)^{p\cdot \p v}\Upsilon_{V,V}(v,v_0)\,.
\end{aligned}
\ee
Note that the evaluation map $\rev_V(v_0\otimes v^*)$ is nonzero only if $v$ and
$v_0$ are colinear. Since $\Upsilon_{V,V}(v_0,v_0)=(-1)^p q^{-2\a(n+\a)}$ and
$\rev_V(v_0\otimes v_0^*)=(-1)^pq^{-2\a}$, we have $\la\th_V\ra=q^{-2\a(n+\a+1)}$ and 
\be
\begin{aligned}
   \th_{V}^*=q^{-2\a(n+\a+1)}\id_{V^*}=
   q^{-2(-\a-n-1))(n+(-\a-n-1)+1)}\id_{V^*}=\th_{V^*}\,.
\end{aligned}
\ee
\end{proof}

\subsection{Modified trace and modified dimension}

In this section we compute the modified trace of typical objects in $\catsl$. 
For $V, W \in \catsl$, let
\be
\Phi_{W,V} = (\id_{V} \otimes \rev_{W})\circ (c_{W,V}\otimes \id_{W^*})
\circ (c_{V,W}\otimes \id_{W^*})\circ (\id_V\otimes \lcoev_{W})\in \End_{\catsl}(V)
\ee
be the map associated to open Hopf link
\be
\Phi_{W,V}
=
\begin{tikzpicture}[baseline=25, scale=1]
%top
\draw[very thick,->] (0,1) to (0,2);  %a
%b
\draw[very thick,over] (-1/2,1) to [out=90, in=180] (0,3/2) to [out=0, in=90] (1/2,1);
%b
% bottom
\draw[very thick,->] (1/2,1) to [out=-90, in=0] (0,1/2) to [out=180, in=-90] (-1/2,1);
%b
\draw[very thick,over] (0,0) to (0,1);  %a
\node[right] at (0,0) {$V$};
\node[right] at (1/2,1) {$W$};
\end{tikzpicture}
\,.
\ee

The following lemma is identical to \cite[Proposition 2.2]{Geer:multi} but with the
addition of the parity refinement, which amounts to the inclusion of the scalar $(-1)^{p'}$. 

\begin{lemma}
\label{lem.Hopf}
Assume $V=V_p(n,\alpha)$ is an irreducible representation and $W=V_{p'}(m,\beta)$.
Then 
\be
\la\Phi_{W,V}\ra = 
(-1)^{p'}q^{-(n+2\a+1)(m+2\b+1)}
 \{\alpha\}
 \{n+\alpha+1\}
 \frac{\{(n+1)(m+1)\}}{\{n+1\}}
 \,.
\ee
\end{lemma}

\begin{proof}
Let $v_0$ be a highest weight vector of $V$. Since $V$ is irreducible,
$\Phi_{W,V}(v_0)=\la \Phi_{W,V}\ra v_0$ and so it is sufficient to consider the
image of $v_0$. Additionally, this simplifies the computation of
$(c_{V,W}\otimes \id_{W^*})(v_0\otimes w\otimes w^*)
=\tau_{V,W}\circ \Upsilon_{V,W}(v_0\otimes w\otimes w^*)$. Due to the evaluation in
the last component of $\Phi_{W,V}$, the precomposition with
$(c_{W,V}\otimes \id_{W^*})$ on $w\otimes v_0\otimes w^*$ is equivalent to
precomposing $\tau_{W,V}\circ \Upsilon_{W,V}$. We additionally observe that
\be
\begin{aligned}
\tau_{W,V}\circ\tau_{V,W}=\id_{V\otimes W}\,, 
&&
\tau_{W,V}\circ \Upsilon_{W,V}= \Upsilon_{V,W}\circ\tau_{W,V}\,,
&&
\Upsilon_{W,V}=\Upsilon_{V,W}\,.
\end{aligned}
\ee
Therefore,
\be
\begin{aligned}
  \Phi_{W,V} &= (\id_{V} \otimes \rev_{W})\circ (c_{W,V}\otimes \id_{W^*})\circ
  (c_{V,W}\otimes \id_{W^*})\circ (\id_V\otimes \lcoev_{W})
\\
&=
(\id_{V} \otimes \rev_{W})\circ ((\tau_{W,V}\circ \Upsilon_{W,V})\otimes \id_{W^*})
\circ ((\tau_{V,W}\circ \Upsilon_{V,W})\otimes \id_{W^*})\circ (\id_V\otimes \lcoev_{W})
\\
&=
(\id_{V} \otimes \rev_{W})\circ (\Upsilon_{V,W}^2\otimes \id_{W^*})
\circ (\id_V\otimes \lcoev_{W})\,.
\end{aligned}
\ee
In the summations below, $0\leq a\leq m$ and $0\leq b,c\leq 1$:
\be
\begin{aligned}
v_0
&\xmapsto[]{\id \otimes \lcoev_W} 
\sum_{a,b,c} 
v_{0} \otimes w_{cba} \otimes w_{cba}^{\ast}
\\
&\xmapsto[]{\Upsilon_{V,W}^2\otimes \id_{W^*}}
\sum_{a,b,c}
 q^{-2(n(\beta+a+b)+\alpha(m-2a-b+c)+2\alpha(\beta+a+b))}
 v_0 \otimes w_{cba} \otimes w_{cba}^*
 \\
 &\xmapsto[]{\id_{V} \otimes \rev_{W}}
 \sum_{a,b,c}
 (-1)^{p'+b+c}
 q^{-2((n+2\alpha)(\beta+a+b)+\alpha(m-2a-b+c))}
 q^{-2(\beta+a+b)}
 v_0 
 \\
 &=
 (-1)^{p'}q^{-2((n+2\a+1)\b+\a m)}
 \left(\sum_{c=0}^1 (-1)^c q^{-2c\a}\right)
 \left(\sum_{b=0}^1 (-1)^b q^{-2b(n+\alpha+1)}\right)
 \left(\sum_{a=0}^m q^{-2a(n+1)}\right) {v_0}
 \\
 &=
 (-1)^{p'}q^{-2((n+2\a+1)\b+\a m)-\a-(n+\a+1)-(n+1)m}
 \{\alpha\}
 \{n+\alpha+1\}
 \frac{\{(n+1)(m+1)\}}{\{n+1\}} {v_0}
 \\
 &=
 (-1)^{p'}q^{-(n+2\a+1)(m+2\b+1)}
 \{\alpha\}
 \{n+\alpha+1\}
 \frac{\{(n+1)(m+1)\}}{\{n+1\}} {v_0} \,.
\end{aligned}
\ee
\end{proof}

\subsection{Projective cover of the tensor unit}

Consider the $\Uq$-module $P$ with generating vector $p_0$ of weight $(0,0)$
which is annihilated by the PBW basis vectors
\be\label{eq.ann}
E_1,F_1, E_1^2E_2, F_1^2F_2, F_1F_2E_2, F_2E_1E_2\,.
\ee
Note that the actions of $F_2$ and $E_2$ anticommute on $\la p_0\ra$ since
$[E_2,F_2]p_0=\frac{K_2-K_2^{-1}}{q-q^{-1}}p_0=0$. Then $P$ is presented in
Figure \ref{fig.proj}.

\begin{figure}[htpb!]
\begin{tikzpicture}
\node (top) at (0,1) {$\bullet$};
\node (bottom) at (0,-1) {$\bullet$};
\foreach \x in {1,2,3}{
\node[] (l\x) at (.2-1.2*\x,0) {$\bullet$};
\node[] (r\x) at (-.2+1.2*\x,0) {$\bullet$};
}
\draw[->,transform canvas={shift={(0,2pt)}}] (l1) -- (l2) node[above,midway] {$F_1$};
\draw[<-,transform canvas={shift={(0,-2pt)}}] (l1) -- (l2) node[below,midway] {$E_1$};
\draw[->,transform canvas={shift={(0,2pt)}}] (l2) -- (l3) node[above,midway] {$F_2$};
\draw[<-,transform canvas={shift={(0,-2pt)}}] (l2) -- (l3) node[below,midway] {$E_2$};
\draw[->,transform canvas={shift={(0,2pt)}}] (r1) -- (r2) node[above,midway] {$E_1$};
\draw[<-,transform canvas={shift={(0,-2pt)}}] (r1) -- (r2) node[below,midway] {$F_1$};
\draw[->,transform canvas={shift={(0,2pt)}}] (r2) -- (r3) node[above,midway] {$E_2$};
\draw[<-,transform canvas={shift={(0,-2pt)}}] (r2) -- (r3) node[below,midway] {$F_2$};
\draw[->] (top) -- (l1) node[above left,midway] {$F_2$};
\draw[->] (top) -- (r1) node[above right,midway] {$E_2$};
\draw[->] (l1) -- (bottom) node[below left,midway] {$E_2$};
\draw[->] (r1) -- (bottom) node[below right,midway] {$F_2$};
\end{tikzpicture}
\caption{The projective cover of the identity. The top-most vertex
  indicates $\la p_0\ra$ of weight $(0,0)$.}
\label{fig.proj}
\end{figure}
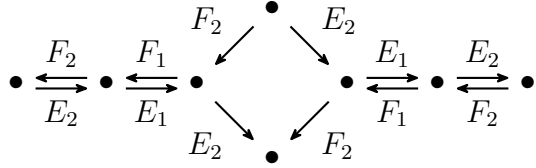

\begin{lemma}
\label{lem.Pproj}
Suppose $(0,\a)$ is typical. Then $V(0,\a)\otimes V(0,\a)^*\cong V(1,-1)\oplus P$.
In particular, $P$ is a projective module.
\end{lemma}

\begin{proof}
By Remark \ref{rem.Vdual}, we may identify $V(0,\a)\otimes V(0,\a)^*$ with
$V(0,\a)\otimes V(0,-\a-1)$. However, the tensor decomposition formula in
Equation \eqref{V0n} does not apply. In particular, there is a vanishing of
highest weight vectors 
\be
\Delta(E_2E_{12}F_2F_{12})(v_0\otimes v_0)\qquad \mbox{and}\qquad 
\Delta(E_2E_{12}F_2F_{12})(v_0\otimes F_2F_1F_2v_0),
\ee
constructed in Lemma \ref{lem.high} for the summands $V(n,\a+\b)=V(0,-1)$ and
$V(n,\a+\b+1)=V(0,0)$, respectively. It is however straightforward to verify that
the highest weight vector $\Delta(E_2E_{12}F_2F_{12})(v_0\otimes F_2v_0)$, for
$V(n+1,\a+\b)=V(1,-1)$ is nonzero and therefore realizes $V(1,-1)$ as a direct
summand of $V(0,\a)\otimes V(0,-\a-1)$.

It remains to show that $P\cong V(0,\a)\otimes V(0,-\a-1)/V(1,-1)$. We prove that
the latter is a quotient of $P$. The claim will then follow as the
dimensions of $P$ and $(V(0,\a)\otimes V(0,-\a-1)/V(1,-1))$ are both $8$. Observe
that the annihilator ideal of 
\be
v=(v_0\otimes F_{12}F_2v_0+\frac{\br{\a}}{q^\a\br{\a+1}}F_{12}F_2v_0\otimes v_0)
+V(1,-1)\in V(0,\a)\otimes V(0,-\a-1)/V(1,-1)
\ee
includes all PBW vectors in Equation \eqref{eq.ann}. Also notice that $v$ is a
generator for the module. Since $p_0$ and $v$ also have weight $(0,0)$, the mapping
$p_0\mapsto v$ defines a surjection of $P$ onto the quotiented tensor product.

The last claim follows from Lemma \ref{lem.proj} and typicality of $(0,\a)$. Since
$V(0,\a)$ and $V(0,\a)^*$ are projective, then their tensor product and its summands
are projective as well.
\end{proof}

\begin{lemma}
\label{lem.cover}
The module $P$ is the projective cover of the tensor unit $\unit$, where the
covering map $p:P\twoheadrightarrow \unit$ is the quotient by the submodule
generated by $E_2p_0$ and $F_2p_0$.
\end{lemma}

\begin{proof}
We showed in Lemma \ref{lem.Pproj} that $P$ is projective. To prove that $(P,p)$
is the projective cover of $\unit$, we show that for any submodule $Q\subset P$
such that $Q+\ker(p) = P$, then $Q=P$. Indeed, the image of $p$ is 1-dimensional
and so $\ker(p)$ is the submodule $P-\la p_0\ra$. For any $Q$ satisfying
$Q+\ker(p) = P$, $Q$ must contain some nonzero multiple of a weight vector
$(1+aF_2E_2)p_0$ for some $a\in \BC$. Notice that $aF_2E_2$ is square nilpotent on
$P$ and so $(1-aF_2E_2)(1+aF_2E_2)p_0=p_0$. Therefore, $Q$ contains $p_0$, a generator
for $P$, so $Q=P$.
\end{proof}

\begin{remark}
We see that $P$
is also the projective cover of both $V(0,0)$ and $\overline{V(0,0)}$ -- the lowest
weight module with lowest weight $(0,0)$; it is not isomorphic to $V(0,0)^*$. The
covering maps are given by taking the quotient by each of the respective submodules
$E_2p_0$ and $F_2p_0$. Thus, $P$ belongs to the exact sequences
\be
\begin{aligned}
0 \to V(0,-1) \to &~P \to V(0,0) \to 0
\,,\\
0 \to \overline{V(0,-1)} \to &~P \to \overline{V(0,0)} \to 0\,.
\end{aligned}
\ee
\end{remark}

\begin{proposition}
\label{prop.tr}
The trace $\mt$ on the ideal of projective modules of $\catsl$ is unique up to a
global scalar.
\end{proposition}

\begin{proof}
We apply \cite[Cor.5.6]{GKPM22}. Since $\catsl$
is a locally finite pivotal $\BC$-linear tensor category with enough projectives,
it is sufficient to show that $\catsl$ is unimodular, as defined in
Section~\ref{sub.cat}. Indeed, the projective cover of $\unit$ is described in
Figure \ref{fig.proj} and its socle is seen to be $\unit$.
\end{proof}

Suppose that $V=V_p(n,\alpha)$ and $W=V_{p'}(m,\beta)$ are both irreducible
representations. Then the cyclicity of the modified trace implies
\be
\begin{aligned}
\mt(\Phi_{W,V})&=\mt(\Phi_{V,W})
\\
(-1)^{p'}
 \frac{\{\alpha\}
 \{n+\alpha+1\}}{\{n+1\}}\md(V)
 &=
 (-1)^{p}
 \frac{\{\b\}
 \{m+\b+1\}}{\{m+1\}}\md(W)\,.
\end{aligned}
\ee
where we have canceled like terms from the open Hopf links computed in Lemma
\ref{lem.Hopf}. We choose a suitable normalization of the modified trace so that
\be
\begin{aligned}
\md(V_p(n,\a))=(-1)^p\frac{\{n+1\}}{\{\alpha\}
 \{n+\alpha+1\}}\,.
\end{aligned}
\ee
By Proposition \ref{prop.tr}, this modified dimension function is essentially unique on $\catsl$.

\section{Links--Gould invariants}
\label{sub.LGcable}

Let $\LG^{(n)}(q^\a,q)$ denote the TQFT invariant of a knot colored by the
representation $V_0(n-1,\a)$ for $n \geq 1$ and normalized 1 at the unknot.
Below, given a knot $K$ in $S^3$, we denote by  
$K^{(n,0)}$ the $(n,0)$-th parallel of a 0-framed knot $K$, and likewise for
$K^{(n,1)}$, or more generally we denote by $K^{(\gamma)}$ the cable of the
0-framed knot $K$ with a fixed pattern $\gamma \in B_n$. 

\begin{proof}(of Theorem~\ref{thm.n0})
Recall the tensor decomposition formula 
\be
V_0(0,\a)^{\otimes n} = \bigoplus_{k + \ell \leq n-1} m^{(n)}_{k,\ell} V_k(k,n\a+\ell)
\ee
which holds for all $\a\notin \BQ$ and any $n\in\BZ_{\geq0}$. As a corollary to
Theorem \ref{thm.n0cable}, 
\be
\begin{aligned}
  \LG_{\knot^{(n,0)}}(q^\a,q)
  &=
  \frac{1}{\md(V(0,\a))}  \mRT_{V(0,\a), \knot^{(n,0)}}
  \\&=
  \frac{1}{\md(V(0,\a))}
  \sum_{k + \ell \leq n-1} m^{(n)}_{k,\ell} \mRT_{V_k(k,n\a+\ell), \knot}
  \\&=
  \sum_{k + \ell \leq n-1} m^{(n)}_{k,\ell} 
  \frac{\md(V_k(k,n\a+\ell))}{\md(V_0(0,\a))}
  \LG_{\knot}^{(k-1)}(q^{n\a+\ell},q)\,.
\end{aligned}
\ee
The result follows from Equation~\eqref{aijn2}.
\end{proof}

\begin{remark}
\label{rem.n1}
There is a similar formula for the $\knot^{(n,1)}$ cable or in fact for a cable
$\knot^{(\gamma)}$ for a fixed pattern $\gamma \in B_n$. We conjecture 
that for all $n \geq 2$, we have 
\be
\label{LGn1}
\begin{aligned}
  \LG^{(1)}_{\knot^{(n,1)}}(q^\a,q) = ~&
  (-1)^{n-1}A^{(n)}_{n-1,0}(q^\a,q) \LG^{(n)}_\knot(q^{n\a},q) \\ & +
\sum_{k=0}^{n-2} (-1)^k\Big(
A^{(n)}_{k,0}(q^\a,q) q^{2\a(n-k-1)} \LG^{(k+1)}_{\knot}(q^{n\a},q) 
\\ & \phantom{ +\sum_{k=0}^{n-2}\Big( }
+
A^{(n)}_{k,n-1-k}(q^\a,q) q^{-2(\a+1)(n-k-1)} \LG^{(k+1)}_{\knot}(q^{n(\a+1)-k-1},q) \Big) \,.
\end{aligned}
\ee
This formula is consistent with the symmetries and specializations of
Equation~\eqref{LGspec}. It has been checked and proven for $n=1,\dots,5$
by an explicit computation. A proof of the formula for all $n$ requires computing
$6j$-symbols in $\catsl$.
\end{remark}

\subsection{Proof of Theorem~\ref{thm.LG}}
Recall first that $\LG^{(1)}$ satisfies Theorem~\ref{thm.LG}. For the
specializations~\eqref{LGspec}, see~\cite{Ishii, Kohli, KPM}, and for the genus
bound~\eqref{LGgenus}, see~\cite{Kohli-Tahar}. 

Next, Theorem~\ref{thm.n0} expresses $\LG^{(n)}$ in terms of $\LG^{(<n)}$ by
splitting the sum on the right hand side of~\eqref{LGn0} as a sum over
$k+\ell \leq n-1, \,\, k \leq n-2$ plus the term with $(k,\ell)=(n-1,0)$. 
Given this, it is easy to see by induction that $\LG^{(n)}$ satisfies the symmetry
in the left hand part of Equation~\eqref{LGspec} if and only if $A^{(n)}_{k,\ell}$
satisfies the symmetry
\be
\label{Asym}
A^{(n)}_{k,\ell}(q^\a,q) = A^{(n)}_{k,n-1-k-\ell}(q^{-\a-1},q)
\ee
for all $n$, $k$ and $\ell$, and the explicit formula~\eqref{aijn} confirms the
latter.

The specialization $\LG^{(n)}_K(1,q)=1$ for all $n$ is well-known for $n=1$~\cite{Ishii}.
For $n \geq 2$, it follows by induction using the fact that the right hand side of
~\eqref{LGn0} is a rational function in $q^\a$ and $q$, and $A^{(n)}_{k,\ell}(q^\a,q)$
is regular at $q^{n\a}=1$ for $\ell>0$ and has computable residue when $\ell=0$, and
\be
m^{(n)}_{k,0} = \binom{n-1}{k} \,.
\ee

The specialization $\LG^{(n)}_{\knot}(q^\a,1)=\Alex_{\knot}(q^{2\a})^2$ is also known
for $n=1$~\cite{Kohli, KPM}. We use induction on $n \geq 1$ to prove the general case.
Assume that $\LG^{(k)}_{\knot}(q^\a,1)=\Alex_{\knot}(q^{2\a})^2$ for all $k < n$. When
$n \geq 2$ the left hand side of~\eqref{LGn0} vanishes when $q=1$, see for example
\cite[Cor.6.10]{HarperAlexander}, whereas the right hand
side is regular since $A^{(n)}_{k,\ell}(q^\a,1) = (-1)^k (k+1) (\br \a/\br {n\a})^2$.
Together with the induction hypothesis, this gives
\be
\label{e1}
0 = \sum_{\substack{k + \ell \leq n-1 \\ k \leq n-2}}
 m^{(n)}_{k,\ell} A^{(n)}_{k,\ell}(q^\a,1) \Alex_{\knot}(q^{2\a})^2 +
 m^{(n)}_{n-1,0} A^{(n)}_{n-1,0}(q^\a,1) \LG^{(n)}_{\knot}(q^\a,1) \,.
\ee
On the other hand, when $\knot$ is the unknot the left hand side of~\eqref{LGn0}
vanishes whereas $\LG^{(k)}_{\mathrm{unknot}}(q^\a,q)=1$, hence we obtain

\be
\label{e2}
0 = \sum_{k + \ell \leq n-1} m^{(n)}_{k,\ell} A^{(n)}_{k,\ell}(q^\a,q) \,.
\ee
Combining~\eqref{e1} and~\eqref{e2}, we deduce that
$\LG^{(n)}_{\knot}(q^\a,1)=\Alex_{\knot}(q^{2\a})^2$
concluding the proof of the right equation of~\eqref{LGspec}. 

Finally, the genus bound in Equation~\eqref{LGgenus} is proven inductively by
combining Theorem~\ref{thm.n0} with the enhancement of the bound
from~\cite{Kohli-Tahar} that is proved in~\cite[Thm.3.3]{NVdV2025} and the
Seifert surface represented in ~\cite[Fig.14,Fig.15,Fig.16]{Kohli-Tahar}. Using
the symmetry from~\eqref{LGspec} satisfied by $\LG^{(n)}$, we want to show that
\be
\label{genusbound}
\maxdeg_{q^\a} \LG^{(n)}_{\knot} (q^{\a},q) \leq 4 n \, \genus(\knot) 
\ee
while we know that
\be
\label{LG1n}
\maxdeg_{q^\a} \LG^{(1)}_{\knot^{(n,0)}} (q^{\a},q) \leq 4 n \, \genus(\knot)-2n+2
\ee
for all $n \geq 1$ from~\cite{Kohli-Tahar,NVdV2025}. But Equation \eqref{LGn0}
can be written alternatively
\be
\begin{aligned}
&\Big( \prod_{\substack{0 \leq m \leq n \\ m \neq 0 \\
m \neq n}} \br {{n\a + m}} \Big) \br {{n}} \br {{\a}} \br {{\a + 1}}
(-1)^{n-1} m^{(n)}_{n-1,0} \, \LG^{(n)}_{\knot}(q^{n\a},q) \\ 
& = \br {{1}} \Big( \prod_{\substack{0 \leq m \leq n }}
\br {{n\a + m}} \Big)\LG^{(1)}_{\knot^{(n,0)}}(q^\a,q) \\ &
- \sum_{\substack{k + \ell \leq n-1 \\ (k,l) \neq (n-1,0)}}
m^{(n)}_{k,\ell} \Big( \prod_{\substack{0 \leq m \leq n \\ m \neq \ell \\
m \neq k + \ell +1}} \br {{n\a + m}} \Big) (-1)^k \br {k+1}
\br \a \br {\a+1} \, \LG^{(k+1)}_{\knot}(q^{n\a +\ell},q) \, ,
\end{aligned}
\ee
so that if we write $a = \maxdeg_{q^\a}\LG^{(n)}_{\knot} (q^{\a},q)$ we get:
\be
\begin{aligned}
(n-2) n + 2 + n a
&\leq \max\left(
n^2 + 4n \, \genus(\knot) - 2n + 2
,
(n-2) n +  4n \, \genus(\knot) +2
\right)
\\
&= 
 (n-2) n +  4n \, \genus(\knot) +2 .
\end{aligned}
\ee
Thus, $a = \maxdeg_{q^\a}\LG^{(n)}_{\knot} (q^{\a},q) \leq 4 \,\text{genus}(\knot)$,
and therefore $\deg_{q^\a}\LG^{(n)}_{\knot} (q^{\a},q) \leq 8 \,\genus(\knot)$.

\subsection{Cabling formula for \texorpdfstring{$V_2$}{V2}} 
\label{sub.thm1}
The $V_2$ polynomial appears in the $(2,1)$-cabling of knots according to the formula
\be
\begin{aligned}\label{eq.Vcable}
V_{1,K^{(2,1)}}(t,\tq)  
&= B^{(2)}_{0,0}(t,\tq) V_{1,K}(t^2\tq^{-1/2},\tq) + 
B^{(2)}_{0,1}(t,\tq) V_{1,K}(t^2\tq^{1/2},\tq) \\
& \hspace{2.5cm} + B^{(2)}_{1,0}(t,\tq) V_{2,K}(t^2,\tq) 
\end{aligned}
\ee
where
\be
\begin{aligned}
\label{A2vals}
B^{(2)}_{0,0}(t,\tq)&=t\tq^{-1/2}\cdot \frac{t(t\tq^{1/2}-1)}{(t^2-1)(t+\tq^{1/2})},
\\
B^{(2)}_{0,1}(t,\tq)&=t^{-1}\tq^{-1/2}\cdot \frac{t(t-\tq^{1/2})}{(t^2-1)(t\tq^{1/2} + 1)},
\\
B^{(2)}_{1,0}(t,\tq)&=\frac{t(\tq+1)}{(t + \tq^{1/2})(t\tq^{1/2} + 1)}\,.
\end{aligned}
\ee
We prove this formula using the spectral decomposition of the $R$-matrix and of
endomorphisms associated to braids. The proof of this formula is nearly identical to
the proof of Remark \ref{rem.n1} aside from the $R$-matrices themselves being
different. Consequently, since $\LG^{(1)}=V_1$ we deduce that $\LG^{(2)}=V_2$. As
noted in the above remark, the proof of this formula for all $n$ requires the
$6j$-symbols, however, the $n=2$ case does not.

Let $R_1$ denote the $R$-matrix used to define $V_1(t,\tq)$ which acts on the tensor
product  $Y_1(t)\otimes Y_1(t)$ of four-dimensional vector spaces $Y_1(t)$ with
basis $(v_1,v_2,v_3,v_4)$. Taking $v_{ij}=v_i \otimes v_j$, the $R$-matrix 
$\mathsf{R}_{1}:=(R_{1}(v_{ij}))_{1 \leq i,j \leq 4}$ is given by 
\begin{center}
\small
\resizebox{\textwidth}{!}{
$
\mathsf{R}_{1}=
\left(
\begin{array}{cccc}
% i=1
- v_{11} 
& 
-t^{-1} v_{21}
& 
-t \tq^{-1} v_{31}
&
-\tq^{-1} v_{41}
\\
% i=2
- v_{12}
+
(t^{-1}-1) v_{21}
& 
t^{-2} v_{22}
& 
- t^2\tq^{-1} v_{32}
+
(1-t^2)\tq^{-1} v_{41}
&
\tq^{-1} v_{42}
\\
% i=3
- v_{13}
+
(t^2\tq^{-1}-1) v_{31}
&
-t^{-2}\tq v_{23}
+
(1-t^{-2}) v_{41}
&
t^2\tq^{-1} v_{33}
&
v_{43}
\\
% i=4
\begin{bmatrix}
- v_{14}
+
(t^{-2}\tq-1) v_{23}
\\
+
(t^2\tq^{-1}-1) v_{32}
+
(t^{-2}+t^2\tq^{-1}-2) v_{41}
\end{bmatrix}
&
t^{-2}\tq v_{24}
+
(t^{-2}-1) v_{42}
&
t^2\tq^{-1} v_{34}
+
(t^2\tq^{-1}-1) v_{43}
&
- v_{44}
\end{array}
\right).$}
\end{center} 

The spectral decomposition of $R_1$ realizes the isomorphism of vector spaces
\begin{align}
   Y_1(t)\otimes Y_1(t) \cong 
Y_1(t^2\tq^{-1/2}) \oplus Y_1(t^2\tq^{1/2}) \oplus Y_2(t^2)
\end{align}
where $Y_2(t^2)$ is an 8-dimensional vector space whose $R$-matrix realizes the knot invariant $V_2(t^2,\tq)$. The eigenvalues corresponding to these
eigenspaces are $t\tq^{-1/2}$, $t^{-1}\tq^{-1/2}$, and $-1$ respectively. We present
the eigenspace decomposition diagrammatically in terms of orthogonal projectors
on $Y_1(t)\otimes Y_1(t)$
\be
\label{eq.idsum}
\begin{tikzpicture}
\draw[->] (0,0) .. controls (.5,.25) and (.5,.75) .. (0,1);
\draw[->] (1,0) .. controls (.5,.25) and (.5,.75) .. (1,1);
\end{tikzpicture}
=
\begin{tikzpicture}
\drawproj{0}{0}{red}   
\end{tikzpicture}
+
\begin{tikzpicture}
\drawproj{0}{0}{green}   
\end{tikzpicture}
+
\begin{tikzpicture}
\drawproj{0}{0}{double}   
\end{tikzpicture}
\ee
where the projector onto the $t\tq^{-1/2}$, $t^{-1}\tq^{-1/2}$, or $-1$ eigenspace
is denoted by the red, green, or doubled projector diagram respectively.
Equivalently, we may write 

\begin{align}
\label{eq.Rsum}
\begin{tikzpicture}
\draw[->] (1,0) to [out=90,in=-90] (0,1);
\draw[over,->] (0,0) to[out=90,in=-90] (1,1);
\end{tikzpicture}
=
t^{-1}\tq^{1/2}
\begin{tikzpicture}
\drawproj{0}{0}{red}   
\end{tikzpicture}
+
t\tq^{1/2}
\begin{tikzpicture}
\drawproj{0}{0}{green}   
\end{tikzpicture}
-
\begin{tikzpicture}
\drawproj{0}{0}{double}   
\end{tikzpicture}
\,.
\end{align}

Fix a long knot diagram $K$ with zero writhe. Apply the following cabling
transformation to $K$ thus defining $K^{(2,0)}$ and $K^{(2,1)}$:
\[
\begin{tikzpicture}
	\draw[->] (1,0) to [out=90,in=-90] (0,1);
	\draw[over,->] (0,0) to [out=90,in=-90] (1,1);
\end{tikzpicture}
\rightsquigarrow
\begin{tikzpicture}
\draw[->] (1.8,0) to [out=90,in=-90] (.3,2);
\draw[->] (1.2,0) to [out=90,in=-90] (-.3,2);
\draw[over,->] (.3,0) to [out=90,in=-90] (1.8,2);
\draw[over,->] (-.3,0) to [out=90,in=-90] (1.2,2);
\end{tikzpicture} ~:
\qquad 
\begin{tikzpicture}
\node[rectangle,inner sep=3pt,draw] (c) at (1/2,1/2) {$K$};
\draw[->] (c.north) to ($(c.north)+(0,.5)$);
\draw[] (c.south) to ($(c.south)+(0,-.5)$);
\end{tikzpicture}
\mapsto
\begin{tikzpicture}
\node[rectangle,inner sep=3pt,draw] (c) at (1/2,1/2) {$K^{(2,0)}$};
\draw[->,transform canvas={shift={(5pt,0)}}] (c.north west) to ($(c.north west)+(0,.5)$);
\draw[->,transform canvas={shift={(-5pt,0)}}] (c.north east) to ($(c.north east)+(0,.5)$);
\draw[transform canvas={shift={(-5pt,0)}}] (c.south east) to ($(c.south east)+(0,-.5)$);
\draw[transform canvas={shift={(5pt,0)}}] (c.south west) to ($(c.south west)+(0,-.5)$);
\end{tikzpicture}
\qquad
\mbox{and}
\qquad 
\begin{tikzpicture}
\node[rectangle,inner sep=3pt,draw] (c) at (1/2,1/2) {$K^{(2,1)}$};
\draw[->,transform canvas={shift={(5pt,0)}}] (c.north west) to ($(c.north west)+(0,.5)$);
\draw[->,transform canvas={shift={(-5pt,0)}}] (c.north east) to ($(c.north east)+(0,.5)$);
\draw[transform canvas={shift={(-5pt,0)}}] (c.south east) to ($(c.south east)+(0,-.5)$);
\draw[transform canvas={shift={(5pt,0)}}] (c.south west) to ($(c.south west)+(0,-.5)$);
\end{tikzpicture}
=
\begin{tikzpicture}
\node[rectangle,inner sep=3pt,draw] (c) at (1/2,1/2) {$K^{(2,0)}$};
\draw[->,transform canvas={shift={(5pt,0)}}] (c.north west) to ($(c.north west)+(0,.5)$);
\draw[->,transform canvas={shift={(-5pt,0)}}] (c.north east) to ($(c.north east)+(0,.5)$);
\draw ($(c.south west)+(5pt,0)$) to[out=-90,in=90] ($(c.south east)+(-5pt,-1)$);
\draw[over] ($(c.south east)+(-5pt,0)$) to[out=-90,in=90] ($(c.south west)+(5pt,-1)$);
\end{tikzpicture}
\,.
\]
Note that the tangle diagram $K^{(2,1)}$ is not balanced (nonzero writhe). However,
the normalized closure of one component of the diagram is balanced as indicated on
the left side of the next equation. From Equation \eqref{eq.Rsum} and Schur's Lemma
we have
\be
\label{eq.sumdiagram}
\begin{aligned}
\begin{tikzpicture}[scale=.8]
\node[rectangle,inner sep=3pt,draw] (c) at (1/2,1/2) {$K^{(2,0)}$};

\coordinate (l) at ($(c.north east)+(-5pt,0)$);
\coordinate (r) at ($(c.north west)+(5pt,0)$);

\path let \p1=($(l)-(r)$), \n1={veclen(\p1)} in
{\pgfextra{\xdef\distlr{\n1}}};

\draw[] ($(c.north east)+(-5pt,0)+(\distlr,0)$) to[out=90,in=-90]
($(c.north east)+(-5pt,1)$);
\draw[over,->] ($(c.north east)+(-5pt,0)$) to[out=90,in=-90]
($(c.north east)+(-5pt,1)+(\distlr,0)$);

\draw[] ($(c.north east)+(-5pt,1)$) arc(180:0:\distlr/2);
\draw[] ($(c.south east)+(-5pt,-1)$) arc(-180:0:\distlr/2);
\draw ($(c.north east)+(-5pt,0)+(\distlr,0)$) to
($(c.south east)+(-5pt,-1)+(\distlr,0)$);

\draw[->] ($(c.north west)+(5pt,0)$) to ($(c.north west)+(5pt,1)$);
\draw ($(c.south west)+(5pt,0)$) to[out=-90,in=90] ($(c.south east)+(-5pt,-1)$);
\draw[over] ($(c.south east)+(-5pt,0)$) to[out=-90,in=90] ($(c.south west)+(5pt,-1)$);
\draw[] ($(c.south west)+(5pt,-1)$) to ($(c.south west)+(5pt,-1)-(0,\distlr/2)$);
\end{tikzpicture}
=&~
t\tq^{-1/2} \cdot V_{1,K}(t^2\tq^{1/2},\tq)~
\underbrace{
\begin{tikzpicture}[scale=.8]
\drawprojc{red};
\end{tikzpicture}
}_{X_1}
+
t^{-1}\tq^{-1/2} \cdot V_{1,K}(t^2\tq^{1/2},\tq)~
\underbrace{
\begin{tikzpicture}[scale=.8]
\drawprojc{green};
\end{tikzpicture}
}_{X_2}
\\&-
V_{2,K}(t^2,\tq)
\underbrace{
\begin{tikzpicture}[scale=.8]
\drawprojc{double};
\end{tikzpicture}
}_{X_3}\,.
\end{aligned}
\ee
Our reference to Schur's lemma follows from property (P1) of $V_1$, proven in
\cite[Thm.1.2]{GHKST}. Note that we also allow ourselves to slide projectors
through crossings, which is a consequence of the Reidemeister 3 move and that for
multiplicity-free eigenspaces, the projectors can be expressed as a linear
combination of 2-braids.

Again by Schur's Lemma, the maps $X_i$ are scalar multiples of the identity on the
vector space $Y_1(t)$. We write $X_i=x_i\cdot \id_{Y_1}$. Equation \eqref{eq.Vcable}
will now follow once the values of the $x_i$ are determined. 

It is in this step that the structure of the proof of Equation \eqref{eq.Vcable}
diverges from the proof of Remark \ref{rem.n1}. In the proof of the remark, the
coefficients $x_i$ are expressed in terms of ratios of modified dimensions. We do
not have that additional structure in this case. Instead we compute these
coefficients by solving a system of equations, which we construct as follows. 

Since any closed diagram evaluates to zero, Equation \eqref{eq.idsum} implies
\begin{align}
\sum_{i=1}^3 x_i=0\,.
\end{align}
Whereas Equation \eqref{eq.Rsum} and the Reidemeister 1 move implies, and similarly
for the inverse crossing,
\begin{align}
  t\tq^{-1/2} x_1+t^{-1}\tq^{-1/2} x_2-x_3=1 \qquad \mbox{and} \qquad
  t^{-1}\tq^{1/2} x_1+t\tq^{1/2} x_2-x_3=1  \,.
\end{align}
This system of three equations in the variables $x_i$ has a unique solution, 
\be
\begin{aligned}
  x_1=
  \frac{t(t\tq^{1/2}-1)}{(t^2-1)(t+\tq^{1/2})}\,,
\qquad
x_2=\frac{t(t-\tq^{1/2})}{(t^2-1)(t\tq^{1/2} + 1)}\,,
\qquad
x_3=\frac{-t(\tq+1)}{(t + \tq^{1/2})(t\tq^{1/2} + 1)}\,.
\end{aligned}
\ee
We now obtain the cabling formula by substituting these values into
Equation \eqref{eq.sumdiagram}.

\subsection{Proof of Theorem~\ref{thm.1}}
Observe that under the change of variables $(t,\tq)=(q^{2\a+1},q^{2})$ we have
\be
\begin{aligned}
  B^{(2)}_{0,0}(q^{2\a+1},q^{2}) &=
  q^{2\a}\frac{\br{\a}\br{\a+1}}{\br{2\a}\br{2\a+1}} = q^{2\a}
A^{(2)}_{0,0}(q^\a,q)\,, \\
B^{(2)}_{0,1}(q^{2\a+1},q^{2}) &=
\frac{\br{\a}\br{\a+1}}{q^{2(\a+1)}
  \br{2\a+1}\br{2\a+2}}= q^{-2(\a+1)} A^{(2)}_{0,1}(q^\a,q)\,,
\\
B^{(2)}_{1,0}(q^{2\a+1},q^{2}) &= [2]\frac{\br{\a}\br{\a+1}}{\br{2\a}\br{2\a+2}}
= -A^{(2)}_{1,0}(q^\a,q)\,.
\end{aligned}
\ee
We showed in \cite{GHKKST} that $\LG^{(1)}$ and $V_1$ are equivalent $R$-matrix
invariants. More precisely, $V_1(q^{2\a+1},q^{2})=\LG^{(1)}(q^\a,q)$. Observe now
that under this evaluation, $V_2(q^{4\a+2},q^2)$ satisfies the same equation as
$\LG^{(2)}(q^{2\a},q)$ in Remark \ref{rem.n1}
\be
\begin{aligned}
V_{1,K^{(2,1)}}(q^{2\a+1},q^2)  
&= B^{(2)}_{0,0}(q^{2\a+1},q^{2}) V_{1,K}(q^{4\a+2},q^2) + 
B^{(2)}_{0,1}(q^{2\a+1},q^{2})V_{1,K}(q^{4\a+3},q^2) \\
& \hspace{2.5cm} + B^{(2)}_{1,0}(q^{2\a+1},q^{2}) V_{2,K}(q^{4\a+4},q^2) \,,
\\
\LG^{(1)}_{K^{(2,1)}}(q^{\a},q)  
&= q^{2\a}
A^{(2)}_{0,0}(q^\a,q) \LG^{(1)}_{K}(q^{2\a},q) + 
q^{-2(\a+1)} A^{(2)}_{0,1}(q^\a,q) \LG^{(1)}_{K}(q^{2\a+1},q) \\
& \hspace{2.5cm} -A^{(2)}_{1,0}(q^\a,q) V_{2,K}(q^{4\a+4},q^{2})\,.
\end{aligned}
\ee
Thus, we have the equality $\LG^{(2)}(q^\a,q)=V_2(q^{2\a+2},q^{2})$.

\begin{corollary}
    Considering $V_1$ as an invariant of links, we obtain the formula for the $(2,0)$ cabling of a knot from Theorems \ref{thm.n0} and \ref{thm.1}:
    \be
    \begin{aligned}
V_{1,K^{(2,0)}}(t,\tq)  
&= \frac{t(t\tq^{1/2}-1)}{(t^2-1)(t+\tq^{1/2})} V_{1,K}(t^2\tq^{-1/2},\tq) + 
\frac{t(t-\tq^{1/2})}{(t^2-1)(t\tq^{1/2} + 1)} V_{1,K}(t^2\tq^{1/2},\tq) \\
& \hspace{2.5cm} + \frac{-t(\tq+1)}{(t + \tq^{1/2})(t\tq^{1/2} + 1)} V_{2,K}(t^2,\tq) \,.
\end{aligned}
    \ee
\end{corollary}

%%%%%%%%%%%%%%%%%%%%%%%%%%%%%%%%%%%%%%%%%%%%%%%%%%%%%%%%%%%%%%%%%%%%%%%%%%%% 
%%%%%%%%%%%%%%%%%%%%%%%%%%%%%%%%%%%%%%%%%%%%%%%%%%%%%%%%%%%%%%%%%%%%%%%%%%%%

\bibliographystyle{hamsalpha}
\bibliography{biblio}
\end{document}